\documentclass[11pt,letterpaper,twoside,reqno]{amsart}
\usepackage{amsmath,amsfonts,amsbsy,amsgen,amscd,mathrsfs,amssymb,amsthm,mathtools,tensor}
\usepackage{siunitx}
\usepackage{algorithm,algorithmic}
\usepackage[shortlabels]{enumitem}
\usepackage{authblk}
\usepackage[foot]{amsaddr}
\usepackage[toc, page]{appendix}
\usepackage[margin=1in]{geometry}
\usepackage[dvipsnames]{xcolor}
\definecolor{MyColor}{HTML}{0047AB}
\usepackage{fancyhdr}
\usepackage{hyperref}
\hypersetup{
colorlinks=true,
linktoc=all, 
linkcolor=blue, 
}
\usepackage{etoolbox}
\makeatletter
\renewcommand{\@secnumfont}{\bfseries}

\makeatother
\patchcmd{\section}{\scshape}{\bfseries}{}{}
\patchcmd{\section}{\normalfont}{\normalfont\color{MyColor}}{}{}
\patchcmd{\subsection}{\normalfont}{\normalfont\color{MyColor}}{}{}
\makeatletter
\def\subsubsection{\@startsection{subsubsection}{3}
\z@{.5\linespacing\@plus.7\linespacing}{-.5em}%
{\normalfont\bfseries}}
\makeatother
\usepackage{graphicx}
\usepackage{float}
\usepackage{caption}
\usepackage{subcaption}
\usepackage{tikz}
\usepackage{booktabs} 
\usepackage{multirow}
\usepackage{tabu} 
\usepackage{orcidlink}
\usepackage[normalem]{ulem} 
\usepackage[color=yellow]{todonotes}

\newtheorem{theorem}{Theorem}[section]
\newtheorem{lemma}[theorem]{Lemma}
\newtheorem{definition}[theorem]{Definition}

\newtheorem{proposition}[theorem]{Proposition} 
\newtheorem{corollary}[theorem]{Corollary}
\newtheorem{remark}[theorem]{Remark}

\makeatletter
\@tfor\next:=abcdefghijklmnopqrstuvwxyzABCDEFGHIJKLMNOPQRSTUVWXYZ\do{
\def\command@factory#1{
\expandafter\def\csname b#1\endcsname{\mathbf{#1}}
\expandafter\def\csname fk#1\endcsname{\mathfrak{#1}}
\expandafter\def\csname bb#1\endcsname{\mathbb{#1}}
\expandafter\def\csname cl#1\endcsname{\mathcal{#1}}
\expandafter\def\csname bcl#1\endcsname{\mathbfcal{#1}}
}
\expandafter\command@factory\next
}

\newcommand{\rmd}{\textnormal{d}}

\newcommand{\R}{\mathbb{R}}

\makeatletter
\@namedef{subjclassname@2020}{\textup{2020} Mathematics Subject Classification}
\makeatother

\begin{document}

\title[Scaled quadratic variation]{Scaled quadratic variation for controlled rough paths and parameter estimation of fractional diffusions}

\author{James-Michael Leahy$^1$\orcidlink{0000-0003-4771-4476}}
\email{j.leahy@imperial.ac.uk}
\author{Torstein Nilssen$^2$}
\email{torstein.nilssen@uia.no}

\address{$^1$ Department of Mathematics, Imperial College London, United Kingdom}
\address{$^2$ Department of Mathematical Sciences, University of Agder, Norway}


\maketitle

\begin{abstract}
We introduce the concept of finite $\gamma$-scaled quadratic variation along a sequence of partitions for paths on a given interval. This concept, with historical roots in the study of Gaussian processes by Gladyshev (1961) and Klein \& Giné (1975), includes the fractional Brownian motion (fBm) with Hurst index $H$, which has finite $1-2H$-scaled quadratic variation. We show that a path that is controlled by a path with finite $\gamma$-scaled quadratic variation in the sense of M. Gubinelli inherits this property, and the corresponding scaled quadratic variation satisfies an Itô-isometry type formula. Moreover, we prove quantitative error bounds that establish a relationship between the convergence rates of the scaled quadratic variation of the controlled path and that of the controlling path. Additionally, we introduce a consistent estimator for the parameter $\gamma$ based on a single sample path, complete with quantitative error bounds. We apply these results to the parameter estimation for fractional diffusions. Our findings specify convergence rates for the estimation of both the Hurst index and the parameters in the noise vector fields. The paper concludes with numerical experiments that substantiate our theoretical findings
\end{abstract}
{\hypersetup{linkcolor=MyColor}
\setcounter{tocdepth}{2}
\tableofcontents}

\section{Introduction}

This work originated from the objective of estimating an unknown parameter $\theta \in \Theta\subset \bbR^M$ from discrete observations of the solution $y$ to a stochastic differential equation:
\begin{equation}\label{eq:rde_drift_intro}
\rmd  y_t = u_t(y_t) \rmd t + \sum_{k=1}^K \sigma_k(y_t;\theta)\rmd Z_t^k  \,,
\end{equation}
where $u: [0,T]\times \bbR^d\rightarrow \bbR^d$ denotes a time-dependent vector field, $\sigma_k(\cdot; \theta) : \R^d \rightarrow \R^d$, $k\in \{1,\ldots, K\}$, denotes a collection of parameterized  vector fields, and $Z :\Omega\times [0,T]\rightarrow \bbR^K$ denotes a driving stochastic path. In particular, $u$ is assumed to be independent of $\theta$ and not estimated. Equations of form \eqref{eq:rde_drift_intro} emerge in various contexts, such as in stochastic fluid dynamics, which we explore in our experimental section.

Our primary example and motivation is when the noise $Z=B^H$ is an $\bbR^K$-valued fractional Brownian motion (fBm) with Hurst index $H \in (\frac13, 1)$. We interpret \eqref{eq:rde_drift_intro} pathwise as a Young differential equation for $H \in (\frac12, 1)$ and as a rough differential equation for $H \in (\frac13, \frac12]$. We introduce estimators for both the Hurst index $H$ of the fBm and the parameter $\theta$, applicable irrespective of whether $H$ is known or unknown. These estimators are derived from numerical approximations of scaled quadratic variations, a concept with a longstanding history in the regularity estimation of Gaussian processes \cite{gladyshev1961new, klein1975quadratic, guyon1989convergence, coeurjolly2001estimating, istas1997quadratic, Begyn}. Our principal contribution is the demonstration of consistency and $\mathbb{P}$-almost sure convergence rates for these estimators in a high-frequency limit over a fixed interval $[0,T]$.

The estimation of the Hurst index and parameters in fBm-driven SDEs has been intensively studied, see e.g., \cite{han2021hurst, 10.1214/20-EJS1685, cont2022rough, kubilius2012rate, kubilius2017parameter} for recent results. Most of the existing literature considers additive noise and the regime $H>\frac{1}{2}$. When considering multiplicative noise, equations usually have only one noise source and can be transformed into the additive noise case \cite{kubilius2017parameter}. In the case $H<\frac{1}{2}$, we are only aware of the work \cite{10.1214/20-EJS1685} where the additive noise case has been considered and it has focused on the estimation of the drift parameters. To our knowledge, the only existing research explicitly addressing multiple-noise sources in the rough regime is the work by A. Papavasiliou, which proceed in an iterative fashion in which one recovers the rough path from the solution assuming the coefficients are known and then proceed with a maximum likelihood estimation to obtain an update for the parameters \cite{papavasiliou2016approximate, zhao2018p, papavasiliou2022inverse}. We present a different approach in this paper.

An early approach towards a pathwise understanding of differential equations driven by the Wiener process can be found in F\"ollmer's paper \cite{Follmer}. Here, the author introduces the notion of quadratic variation \emph{along a sequence of partitions} $\{ \pi_n \}$ of a path $x:[0,T]\rightarrow \bbR$:
$$
\langle x \rangle_t := \lim_{n \rightarrow \infty} \sum_{ [t_i,t_{i+1}] \in \pi_n \cap [0,t]}  (x_{t_{i+1}} - x_{t_i })^2.
$$
It is well known that for almost all sample paths of the Wiener process, $W$, 
$$
\langle W \rangle_t = t\,,
$$
if, for example, the partition is uniform $\pi_n = \{ \frac{iT}{n} \}_{i=0}^n$. 
Moreover, by It\^{o}'s formula, for any smooth function $f : \R^d \rightarrow \R$,
\begin{equation} \label{quadratic variation formula}
\langle f(y) \rangle_t = \sum_{k=1}^K \int_0^t |\sigma_k[f](y_r;\theta)|^2  \rmd r\,,
\end{equation}
where $y$ is the solution of \eqref{eq:rde_drift_intro} and $\sigma_k[f]$ denotes the vector field action of $\sigma_k$ on $f$. This observation can be used as a starting point to fit the observed values $y$ to the model \eqref{eq:rde_drift_intro}; one chooses a suitable cost function in order to minimize the distance between both sides of \eqref{quadratic variation formula} when $f$ ranges over a set of smooth functions. 

The notion of quadratic variation along a sequence of partitions does not directly extend to sources of noise that behave differently than the Wiener process under scaling. For example, in the case of 1-dimensional fBm $B^H$ (see below for a precise definition), we have
$$
\langle B^H \rangle_t = \left\{
\begin{array}{ll}
\infty & \textrm{ for } H < \frac12 \\ 
t & \textrm{ for } H = \frac12 \\ 
0 & \textrm{ for } H > \frac12 \\ 
\end{array} \right. ,
$$
see \cite[p.4]{BFGMS} (in fact, this also follows from our Lemma \ref{prop:sv_properties}). 

To address this limitation, we introduce a scaling parameter in the definition of quadratic variation, termed the $\gamma$-scaled quadratic variation. For a path $x:[0,T]\rightarrow \bbR^K$, it is defined as
\begin{equation} \label{eq:first qv}
\langle x \rangle_t^{(\gamma)} := \lim_{n \rightarrow \infty} \sum_{ [t_i,t_{i+1}] \in \pi_n \cap [0,t]} (t_{i+1} - t_i)^{\gamma} (x_{t_{i+1}} - x_{t_i })^2\,.
\end{equation}
The only non-trivial scaling parameter for fBm with Hurst index $H$ is $\gamma = 1-2H$. A general class of Gaussian processes for which the above limit exists for $\gamma\ne 0$ have been studied in the literature \cite{gladyshev1961new, klein1975quadratic}. Moreover, our concept of $\gamma$-scaled quadratic variation can be understood as a normalized quadratic variation along a potentially irregular partition, which has been used to estimate the regularity of Gaussian processes discussed in \cite{istas1997quadratic, Begyn}. 
For our purposes, we need the convergence in \eqref{eq:first qv} \emph{for all times} $t$ on some set in $\Omega$, a feature which has not been explored in any of these references. 

Our paper explores the following applications of scaled quadratic variation: 
\begin{enumerate}
\item We demonstrate that the concept of scaled quadratic variation is compatible with the rough path framework. Specifically, paths that are \emph{controlled} (in the Gubinelli sense) by a path with finite scaled variation also retain finite scaled variation. In the special case of fBm driving paths, we identify an appropriate analogue of \eqref{quadratic variation formula} for equation \eqref{eq:rde_drift_intro} interpreted in the pathwise sense. This is expressed as:
\begin{equation} \label{eq:scaled variations isometry}
\langle f(y) \rangle_T^{(1-2H)} = \sum_{k=1}^K \int_0^T |\sigma_k[f](y_r)|^2 \rmd r,
\end{equation}
\item \label{item:Hest} For an observed solution of \eqref{eq:rde_drift_intro} sampled at discrete time points, we propose an estimator for the scaling parameter $\gamma$. Notably, for the fBm case with $H > \frac12$, this estimator aligns with the one found in \cite{kubilius2017parameter}.
\item \label{item:Thetaest} In the special case fBm, we introduce an estimator for $\theta$:
$$
\hat{\theta} = \arg \min_{\theta \in \Theta} \sum_{f \in \mathbb{F}} \left| \langle f(y) \rangle_T^{(1-2H)} - \sum_{k=1}^K \int_0^T |\sigma_k[f](y_r;\theta)|^2 \rmd r \right|^2 \,,
$$
where $\mathbb{F}$ is a suitable set of functions $f: \R^d \rightarrow \R$.
\item When the dependence of the model on parameter $\theta$ is linear, we establish convergence rates for the estimators described in \ref{item:Hest} and \ref{item:Thetaest}.
\item To illustrate the practical utility of our method, we conduct and present a series of numerical experiments.
\end{enumerate}

Let us emphasize that, in points (3) through (5) above, the present paper will explore parameter estimation only when the dependence on the parameters $\theta$ is \emph{linear}. Although we expect to get convergence of the estimators also in the non-linear setting, it is not clear how to obtain convergence rates in this setting. This comes from the fact that, in the linear setting, the estimator has an analytic expression (under some non-degeneracy condition). 

In the fBm setting, the equality presented in \eqref{eq:scaled variations isometry} agrees with \cite[Theorem 5.12]{liu2020discrete}, under the condition that $H \in (\frac14, \frac34)$. A key point to emphasize is that our approach is not limited to $H < \frac34$. The authors in \cite{liu2020discrete} expand this equality to include a general Gaussian process as the driving noise and further establish a central limit theorem for convergence. This theorem serves as a crucial foundation for developing central limit theorems for parameter estimation.

The concept of $p$-variation along a sequence is closely related to the approach of this paper. For a path $x:[0,T]\rightarrow \bbR^K$, it is defined as
$$
[x]^{(p)}_t  := \lim_{n \rightarrow \infty} \sum_{ [t_i,t_{i+1}] \in \pi_n \cap [0,t]}  |x_{t_{i+1}} - x_{t_i }|^p \,.
$$
In particular, it holds that $[B^{H}]^{1/H} = c_H t$ (see \cite{Rogers}) where $c_{H} = \mathbb{E}(|B_1|^{1/H})$. The study in \cite{cont2019pathwise} further explores this notion, developing a pathwise integration theory that includes a change of variables formula and an It\^{o} isometry.

In a 1-dimensional context, the paper \cite{powervariation} investigates $p$-variation estimators of pathwise integrals on the form $\int_0^t u_s dB_s^H$ where $u$ is regular enough to understand the integration in the sense of Young. It is shown that 
$$
n^{-1+pH}\sum_{i=0}^{\lfloor tn \rfloor } \left| \int_{i/n}^{(i+1)/n} u_s dB_s^H \right|^p \rightarrow \mathbb{E}(|B_1|^p) \int_0^t |u_s|^p ds
$$
in the uniform convergence in probability. When $p=2$ and on the uniform grid this corresponds exactly to our scaling $\gamma = 1-2H$.
Also in the 1-dimensional context, the paper \cite{liu2023power} investigates $p$-variations of controlled rough paths in scenarios where the driving noise is fBm, leading to central limit type theorems. Additional results in this vein are also presented in the thesis \cite{zhou2018parameter}.

For differential equations driven by a 1-dimensional fBm, the parameter estimation methodologies of this paper could be derived using $p$-variation along a sequence of partitions. It is, however, not clear how the approach could be used in the presence of several sources of noise. Let us also emphasize that the pathwise nature of the present paper yields estimators based on a single observation.

The paper is organized as follows. In Section 2, we introduce the notation and basic preliminaries. This includes an overview of rough paths, controlled rough paths, rough differential equations, and fractional Brownian motion. In Section 3, we introduce scaled quadratic variation and covariation, illustrating its inheritance by controlled rough paths. We also prove a bound linking the convergence rates of these paths to those of the controlling path. Furthermore, this section presents an estimator for the scaling exponent and provides a bound for its convergence rate in relation to the scaled quadratic variation of the path. Section 4 focuses on the convergence rates of the scaled quadratic variation of fBm in the supremum norm and uses this finding in conjunction with the results in Section 3 to deduce the convergence rates of the scaled quadratic variation of solutions of fractional diffusions. Section 5 is devoted to the estimation of the Hurst parameter and the parameters of a linear parameterized vector field from a solution path of a fractional diffusion. We demonstrate the estimators' consistency and detail $\bbP$-a.s. convergence rates achieved over a dyadic sequence of partitions. In Section 6, we conclude with numerical experiments, which are presented in four examples: i) a 1-dimensional non-linear diffusion, ii) a 2-dimensional linear diffusion, iii) a 2-dimensional non-linear diffusion, and iv) a 2-dimensional Euler system with Lie transport noise (fluid RPDE).

\addtocontents{toc}{\protect\setcounter{tocdepth}{-1}}

\subsection*{Acknowledgments}
We would like to thank our colleagues for their invaluable insights and discussions, especially Darryl Holm, Patrick Kidger, Yanghui Liu, Anastasia Papavasiliou, and Samy Tindel. We thank Suprio Bhar, Purba Das and Barun Sarkar for pointing out an error in a previous version of the paper. Our gratitude also goes to Aythami Bethencourt de León for their contributions in testing and developing parts of the code-base used in the simulation of rough differential and partial differential equations in our experiments section. JML wishes to acknowledge the support received from the US AFOSR Grant FA8655-21-1-7034, which has been instrumental in this research.

\addtocontents{toc}{\protect\setcounter{tocdepth}{3}}

\section{Notation and preliminaries}

\subsection{Notation and conventions}

Given $T>0$, a partition $\pi$ of the interval $[0,T]\subset \bbR$ is defined to be the set of divisions of an interval $[0,T]$ into a finite number of adjacent intervals $[t_{i-1}, t_{i}]$, $i\in \{1,\ldots, \#\pi\}$, where $0=t_0<t_1<\cdots <t_{\# \pi}=T$. The mesh of $\pi$ is defined by $|\pi|=\max_{1\le i\le \#\pi}(t_{i}-t_{i-1})$. The set of all partitions of the interval $[0,T]$ is denoted by $\clD([0,T])$. A partition $\pi \in \clD([0,T])$ is said to be uniform if $t_i-t_{i-1}=|\pi|$ for all $i\in \{1,\ldots, \#\pi\}$ so that $\#\pi|\pi| =T$.
Given $t\in [0,T]$, denote by $\pi\cap [0,t]$ the set of adjacent intervals of $\pi$ that lie in $[0,t]$. Note that $\pi\cap [0,t]$ is not necessarily a partition of $[0,t]$ itself unless $t=t_i$ for some $i\in \{1,\ldots, \#\pi\}$. A sequence of partitions $\Pi=\{\pi_n\}_{n \in \mathbb{N}}$ is said to have vanishing mesh if $\lim_{n\rightarrow \infty} |\pi_n|=0$.

Given $d,m\in \bbN$ and a measurable set of $U\subset \bbR^m$ and a measurable function
$f:U\rightarrow \bbR^d$, we define 
$$\|f\|_{\infty}:=\underset{x\in U}{\operatorname{ess \, sup}}|f(x)|<\infty.$$ 
Denote by $C(U;\bbR^m)$  the space of continuous functions from $U$ to $\bbR^m$ (not necessarily bounded). Given  $l\in \bbN$, denote by $C^l(\bbR^d;\bbR^m)$ the space of $l$-times continuously differentiable functions from $\bbR^d$ to $\bbR^m$.
Denote by  $BV([0,T];\bbR^d)$ the  space of functions $f:[0,T]\rightarrow \bbR^d$ with bounded variation
$$[f]_{\textnormal{BV}}:=\sup_{\pi\in \clD([0,T])}\sum_{i=0}^{\#\pi -1} |\delta f_{t_i t_{i+1}}|<\infty\,,$$
where $\delta f_{st}:=f_t-f_s$ for $s,t\in [0,T]$.
Given  $\alpha\in (0,1)$, denote by $C^{\alpha}([0,T];\bbR^d)$ the space of functions $f: [0,T]\rightarrow \bbR^d$ with finite $\alpha$-H\"older semi-norm
$$
[f]_{\alpha}:=\sup_{s,t\in [0,T]}\frac{|\delta f_{st}|}{|t-s|^{\alpha}}<\infty\,,
$$
where we use the convention $0/0:=0$.  Given  $\alpha>0$, denote by $C^{\alpha}_2([0,T];\bbR^d)$ the space of functions $f: \Delta_T\rightarrow \bbR^d$ with finite $\alpha$-H\"older semi-norm
$$
[f]_{\alpha}:=\sup_{s,t\in [0,T]}\frac{|f_{st}|}{|t-s|^{\alpha}}<\infty\,,
$$
where we have defined $\Delta_T := \{ (s,t) \in [0,T]^2 : s \leq t \}$.
For $f\in C^{\alpha}([0,T];\bbR^d) \cup C_2^{\alpha}([0,T];\bbR^d)$, we denote
$$
\|f\|_{\alpha}=\|f\|_{\infty} + [f]_{\alpha}\,.
$$
Since the domain and range of a function will always be clear from the context, we will not indicate this in the notation of the norms or semi-norms.  Moreover, if the range is $\bbR$, then we will drop the range from the notation of the space as well.

Given a vector field $\sigma: \bbR^d\rightarrow \bbR^d$ and a differentiable scalar-valued function $f:\bbR^d\rightarrow \bbR$, denote by $\sigma[f]:\bbR^d\rightarrow \bbR$ the function defined by
$$
\sigma[f](x)=\sum_{j=1}^d\sigma^j(x)\partial_{x_j}f(x)\,, \quad x\in \bbR^d\,.
$$
Similarly, given a vector field $\sigma_1: \bbR^d\rightarrow \bbR^d$ and differentiable vector field $\sigma_2: \bbR^d\rightarrow \bbR^d$, denote by $\sigma_1[\sigma_2]: \bbR^d\rightarrow \bbR^d$ the vector field defined by 
$$
\sigma_1[\sigma_2]^i(x) = \sum_{j=1}^d \sigma_1^j(x) \partial_{x_j} \sigma_2^i(x)\,, \quad x\in \bbR^d\,.
$$
For a given matrix $A\in \bbR^{d\times k}$, denote by $\|A\|_2$ the Frobenius matrix norm of $A$ and by $\kappa(A) = \|A\|_2 \| (A^TA)^{-1} A^{T}\|_2$ the condition number of $A$.
Throughout the paper, we will indicate that a constant $C>0$ depends on quantities $q_1,\ldots, q_n$ by writing $C=C(q_1,\ldots, q_n)$.

\subsection{Rough paths, controlled rough paths, and rough differential equations}

\begin{definition}
Let $K \in\bbN $ and $\alpha \in (\frac13,\frac{1}{2}]$. An $\alpha$-H\"older-rough path is a  pair 
$$
\bZ=(z, \mathbb{Z}) \in C^{\alpha}([0,T];\R^{K}) \times C^{2 \alpha}_2 ([0,T]; \R^{K\times K}) =:\mathscr{C}^{\alpha}([0,T];\bbR^K)
$$
that satisfies Chen's relation: 
\begin{equation}\label{eq:chen_relation}
\mathbb{Z}_{s t}= \mathbb{Z}_{s u} + \delta z_{su} \otimes  \delta z_{u t} + \mathbb{Z}_{u t},   \quad  \forall (s,u,t) \in [0,T]^3 \quad  \text{such that} \quad  s \leq u \leq t \, .
\end{equation}
We define $[\bZ]_{\alpha} := [z]_{\alpha} + \sqrt{[\mathbb{Z}]_{2 \alpha}}$.
\end{definition}

If $z\in C^{\alpha}([0,T];\bbR^K)$ with $\alpha \in (\frac12,1]$, then $\bbZ\in C_2^{2\alpha}([0,T];\bbR^{K\times K})$ defined by the Young integral (\cite{Young36})
$$
\mathbb{Z}_{st}^{l,k} = \int_s^t \delta z_{sr}^l \rmd z_r^k\, ,
$$
satisfies the second identity in \eqref{eq:chen_relation} and can be understood as a continuous operation from  $C^{\alpha}([0,T];\R^{K})$ to $C_2^{2\alpha}([0,T];\R^{K\times K})$. If $\alpha \in (1/3,1/2]$, then, in general, it is not possible to define a continuous operation to define $\bbZ$ canonically, but for certain paths with a probabilistic structure, one can construct many $\bbZ$ satisfying \eqref{eq:chen_relation} (see, e.g., Theorem \ref{thm:fBm_lift} below). To simplify the notation below, we will abuse notation and define $\mathscr{C}^{\alpha}([0,T];\bbR^{K})=C^{\alpha}([0,T];\bbR^K)$ for $\alpha\in (1/2,1]$ (i.e., $\bZ=z$).

Next, we recall the notion of a controlled rough path \cite{Gub04}.
\begin{definition}
Let $z\in C^{\alpha}([0,T];\bbR^K)$ for some $\alpha \in (\frac13,1]$. A path $y\in C^{\alpha}([0,T];\bbR^d)$ is said to be controlled by $z$ if there exists a continuous path $y'\in  C^{\alpha}([0,T];\bbR^{d\times K})$ such that $y^{\sharp}: \Delta_T \rightarrow \bbR^d$ defined by
$$
y^{\sharp}_{st} := \delta y_{st} - y'_s\delta z_{st}\,, \quad (s,t)\in \Delta_T\,,
$$
satisfies $y^{\sharp}\in  C_2^{2 \alpha}([0,T]; \bbR^d)$. The path $y'$ is called the Gubinelli derivative and $y^{\sharp}$ is called the remainder. We denote by $\mathscr{D}^{2 \alpha}_z([0,T];\bbR^d)$ the vector space of paths controlled by $z$.
\end{definition}
\begin{proposition}\label{prop:chain_rule}
Let $z\in C^{\alpha}([0,T];\bbR^K)$ for $\alpha \in (\frac13,1]$ and $f\in C^2(\bbR^d)$. If  $y \in \mathscr{D}^{2 \alpha}_z([0,T];\bbR^d)$, then $f(y)\in \mathscr{D}^{2 \alpha}_z([0,T];\bbR^d)$ with Gubinelli derivative $f(y)'=Df(y) y'$.
\end{proposition}

The prototypical example of a controlled rough path is the solution of a rough differential equation.

\begin{definition}\label{def:rde}
Let $\bZ\in \mathscr{C}^{\alpha}([0,T];\bbR^K)$ with $\alpha\in (1/3,1]$.
Let $\beta\in C([0,T] \times \bbR^d;\bbR^d)$ and $\sigma=(\sigma_{k})_{1\le k\le K}\in C^{\lfloor 1/\alpha\rfloor}(\bbR^d;\bbR^{d\times K})$.  A path $y\in C^{\alpha}([0,T];\bbR^d)$ is said to solve the  differential equation
\begin{equation}\label{eq:rde}
\rmd y_t = u_t(y_t)\rmd t + \sum_{k=1}^K \sigma_k(y_t) \rmd \bZ^k
\end{equation}
in the case if the 2-parameter map $y^{\sharp} : \Delta_T \rightarrow \R^d$ defined by
$$
y_{st}^{\sharp} : = \delta y_{st}-\int_s^tu_r(y_r)\rmd r -  \sum_{k=1}^K \sigma_k(y_s) \delta z_{st}^k
$$
belongs to $C_2^{2 \alpha}([0,T];\R^d)$, and when $\alpha\in (1/3,1/2]$, the 2-parameter map $y^{\natural} : \Delta_T \rightarrow \R^d$ defined by
$$
y_{st}^{\natural} : = \delta y_{st} -\int_s^tu_r(y_r)\rmd r-  \sum_{k=1}^K \sigma_k(y_s) \delta z_{st}^k - \sum_{k,l = 1}^K \sigma_l [\sigma_k](y_s) \mathbb{Z}_{st}^{l,k}    
$$
belongs to $C_2^{3 \alpha}([0,T];\R^d)$.
\end{definition}
\begin{remark}
In particular, due to the continuity of $\beta,\sigma, D\sigma$ and the boundedness of $y$, if $y$ solves \eqref{eq:rde}, then $y\in \mathscr{D}^{2 \alpha}_z([0,T];\bbR^d)$ with Gubinelli derivative $y'=\sigma(y)$.
\end{remark}

Sufficient conditions on  $\beta$ and $\sigma$ for the existence and uniqueness of solutions of \eqref{eq:rde} can be found, e.g.,  in \cite[Chapter 8]{FrizHairer}.

Solutions of rough partial differential equations, for example, in the sense of unbounded rough drivers, are controlled paths in function spaces, and upon testing a solution by a sufficiently regular test function, the resulting function is a scalar-valued controlled rough path. See Example \ref{sec:Euler_system}.

\subsection{Fractional Brownian motion}
Given $K\in \bbN$,  an $\bbR^K$-valued fractional Brownian motion (fBm)  with Hurst parameter $H\in (0,1)$  supported on a probability space  $(\Omega,\mathcal{F},\mathbb{P})$ is a centered Gaussian process  $B = (B^1, \dots, B^K): [0,\infty) \times \Omega \rightarrow \R^K$ with covariance satisfying
$$
\mathbb{E}(B_t \otimes B_s)  =  \frac12 ( t^{2H} + s^{2H} - |t-s|^{2H}) I_{K \times K}\,, \;\; (s,t)\in [0,T]^2\,.
$$
If $H = \frac12$, then $B$ is an $\bbR^K$-valued standard Brownian motion. Moreover, we note that
\begin{itemize}
\item $B$ has stationary increments, that is, $B_t - B_s \sim \mathcal{N}(0,|t-s|^{2H} I_{K \times K} )$ and $B_0=0$;
\item for all $c\in \bbR$, $\{ B_{ct} \}_{t \geq 0} \sim \{ c^H B_{t} \}_{t \geq 0}$;
\item if $B^1$ and $B^2$ are scalar-valued independent fBms and $a,b \in \R$, then 
\begin{equation} \label{eq:sum_of_fBm}
\bar{B}_t := (a^2 + b^2)^{- \frac12} (aB_t^1 + b B_t^2)
\end{equation}
is a fBm;
\item for all $u \leq v \leq s \leq t$, 
\begin{equation} \label{eq:cov_neg_pos}
\mathbb{E}(\delta B_{uv} \delta B_{st}) \in 
\left\{
\begin{array}{ll}
(-\infty, 0] & \textrm{ if } H \leq \frac12 \\
\big[0,\infty) & \textrm{ if } H > \frac12 \\
\end{array}
\right. . 
\end{equation}
If $H\leq \frac 12$, then this is proved in \cite[Lemma 10.8]{FrizHairer}. The case $H > \frac12$ can be proved in the same way as \cite[Lemma 10.8]{FrizHairer};
\item for all $s \leq u \leq v \leq t$,
\begin{equation} \label{eq:cov_subset}
\mathbb{E}( \delta B_{uv} \delta B_{st})  \leq 
\left\{
\begin{array}{ll}
|t-s|^{2H} & \textrm{ if } H \leq \frac12 \\
|t-s| & \textrm{ if } H > \frac12 ,\\
\end{array}
\right. .
\end{equation}
If $H\leq \frac 12$, then this is proved in \cite[Lemma 10.8]{FrizHairer}. The case $H > \frac12$ is straightforward; 
\item   $\bbP$-a.s.,  $B\in C^{\alpha}([0,T];\bbR^K)$ for any $\alpha < H$.
\end{itemize}

The following result from \cite[Theorem 10.4, Corollary 10.10 and Example 10.11]{FrizHairer} shows that if $H \in (\frac13, \frac12]$, then $B$ can be lifted to a rough path.

\begin{theorem} \label{thm:fBm_lift}
Let $B$ denote an $\bbR^K$-valued fractional Brownian motion supported on a probability space $(\Omega, \clF, \bbP)$ with Hurst parameter $H\in (\frac13, \frac12]$. Given $T>0$, there exists a measurable mapping $\bbB: \Delta_T\times \Omega \rightarrow \bbR^{K\times K}$ and a null set $N\in \clF$ such that $\bB=( \delta B,\bbB): \Delta_T \times \Omega \rightarrow \bbR^K\times \bbR^{K\times K}$ satisfies
$$
\bB(\omega) = ( \delta B(\omega),\mathbb{B}(\omega)) \in \mathscr{C}^{\alpha}([0,T];\R^K), \;\; \forall \omega\in N^c\,.
$$
for any $\alpha<H$.
Moreover, there exists a random variable $M \in \cap_{p \geq 1} L^p(\Omega; \clF, \bbP)$ such that
$$
|\delta B_{st}| \leq M|t-s|^{\alpha} , \qquad |\mathbb{B}_{st}| \leq M|t-s|^{2\alpha} , \quad \forall (s,t)\in \Delta_T\,.
$$
Lastly,  $\bbB$ is a limit in the rough path topology of $\bbB^n_{st}=\int_s^t \int_s^{t_1} \rmd B^{n}_{r} \rmd B^n_{t_1}$, where $B^n$ is a piecewise linear interpolation of $B$ along a dyadic sequence of partitions of the interval $[0,T]$.
\end{theorem}

In particular, Definition \ref{def:rde} defines a pathwise solution of stochastic differential equations driven by fBm. Clearly, there are infinitely many lifts of fBm.  For example, given $f\in C^{2\alpha}([0,T];\bbR^{K\times K})$, $(B,\bbB^f)\in \mathscr{C}^{\alpha}([0,T];\R^K)$, where  $\bbB^f = \bbB + \delta f$ and $\bbB$ is as in Theorem \ref{thm:fBm_lift}. The results below are independent of the choice of the rough path lift.

\section{Scaled quadratic variation}

Let $\Pi= \{\pi_n\}$ denote a sequence of partitions of the interval $[0,T]$ with vanishing mesh. Given $\gamma\in \bbR$, $x,z:[0,T]\rightarrow \bbR$, $d\in \bbN$, $y:[0,T]\rightarrow \bbR^d$, $n\in \bbN$, and $t\in [0,T]$, we define
\begin{align*}
\langle x\rangle_t^{(\gamma),n}&=\sum_{ [t_i,t_{i+1}] \in \pi_n \cap [0,t]} (t_{i+1} - t_i)^{\gamma} (\delta x_{t_i t_{i+1}})^2\,,\qquad 
\langle \langle y\rangle\rangle^{(\gamma),n}_t=\sum_{i=1}^d\langle y^i\rangle_t^{(\gamma),n}\,,\\
\langle x,z\rangle_t^{(\gamma),n}&=\sum_{ [t_i,t_{i+1}] \in \pi_n \cap [0,t]} (t_{i+1} - t_i)^{\gamma} (\delta x_{t_i t_{i+1}}) (\delta z_{t_i t_{i+1}})\,.
\end{align*}
We note that $\langle x\rangle^{(\gamma),n}, \langle \langle  y\rangle \rangle^{(\gamma),n}, \langle x,z\rangle^{(\gamma),n} \in \operatorname{BV}([0,T];\bbR)$ since they are all increasing functions of $t$.

\begin{definition} \label{def:scaled variation}
Let $\Pi= \{\pi_n\}$ denote a sequence of partitions of $[0,T]$ with vanishing mesh. 
A path $x:[0,T] \rightarrow \R$  is said to have finite \emph{$\gamma$-scaled covariation} along $\Pi$ if 
$
\langle x\rangle_t^{(\gamma)} := \lim_{n \rightarrow \infty} \langle x\rangle_t^{(\gamma),n}
$
exists for all $t\in [0,T]$ and $\langle x\rangle^{(\gamma)} \in \textnormal{BV}([0,T];\R)\cap C([0,T];\R)$. 

Given $d\in \bbN$,  we denote by $\mathcal{V}_{\Pi}^{\gamma}([0,T];\bbR^d)$ the set of all paths $x=(x^1,\ldots, x^d):[0,T] \rightarrow \R^d$ such that 
$
\langle x^i,x^j\rangle_t^{(\gamma)} := \lim_{n \rightarrow \infty} \langle x^i,x^j\rangle_t^{(\gamma),n}
$
exists for all $i,j\in \{1,\ldots, d\}$ and $t\in [0,T]$, and $\langle x^i,x
^j\rangle^{(\gamma)} \in \textnormal{BV}([0,T];\R)\cap C([0,T];\R)$, for all $i,j\in \{1,\ldots, d\}$.  If $x\in \mathcal{V}_{\Pi}^{\gamma}([0,T];\bbR^d)$, we write $\langle \langle x\rangle \rangle^{(\gamma)}=\sum_{i=1}^d \langle x^i\rangle^{(\gamma)}$. 
\end{definition}
\begin{remark}
The reason for requiring the covariation to exist for $d$-dimensional paths is because, in general, it is not true that if  $x:[0,T]\rightarrow \mathbb{R}$ and $z:[0,T]\rightarrow \mathbb{R}$ have finite $\gamma$-scaled variation along $\Pi$, then 
$
\langle x,z\rangle_t^{(\gamma)} := \lim_{n \rightarrow \infty} \langle x,z\rangle_t^{(\gamma),n}
$
exists for all for all $t\in [0,T]$, as proved in \cite[Prop. 2.7]{SCHIED2016974}. Note that if $x,z,x+z,x-z$ have finite $\gamma$-scaled variation along $\Pi$, then since the polarization identity holds for each $n$, viz
$$
\langle x, z \rangle_t^{(\gamma),n}=\frac14 \left( \langle x + z  \rangle^{(\gamma),n} - \langle x - z  \rangle^{(\gamma),n} \right)\,,
$$
we have $(x,z)\in \mathcal{V}_{\Pi}^{\gamma}([0,T];\bbR^2)$. 
\end{remark}

\bigskip

In Section \ref{sec:fBm_convergence_rates}, we show that for $\mathbb{P}$-a.e.\, $\omega\in \Omega$, the sample path of an $\bbR^k$-valued fBm satisfies 
$\langle B^k(\omega), B^l(\omega) \rangle^{(1-2H)}_t = \delta_{k=l} t$ for all $t\in [0,T]$ along a family of uniform partitions (which includes $\# \pi_n = n$ and $\# \pi_n = 2^{n}$). 

Previous studies such as \cite{klein1975quadratic} and \cite{liu2020discrete} have reported similar results for a more extensive class of Gaussian processes $z = (z^1, \dots, z^K)$. However, the extent to which these results can be directly applied within our framework remains uncertain, as we shall now see. 

The work \cite{klein1975quadratic} demonstrates that when $z$ is a so-called Gladyschev process, for a fixed $t$, the convergence of $\langle z^k(\omega),z^l(\omega) \rangle^{(\gamma),n}_t$ occurs for almost all $\omega \in \Omega$, including convergence rates. In the particular case of a 1-dimensional fBm, $B$, (which is an example of a Gladyshev process), \cite[Theorem 2]{klein1975quadratic} shows that 
$$
\lim_{n \rightarrow \infty} \langle B \rangle^{(1-2H),n}_t = t, \quad \mathbb{P}-a.s.
$$
Yet, this does not fulfill our Definition \ref{def:scaled variation}, which requires that $\langle z^k(\omega),z^l(\omega) \rangle^{(\gamma),n}_t$ converges \emph{for all} $t\in [0,T]$ and \emph{for all} $\omega$ belonging to some set of full measure.

The work \cite{liu2020discrete} shows that $\langle z^k(\cdot),z^l(\cdot) \rangle^{(\gamma),n}_t$ converges to $t$ in finite-dimensional distributions for the class of Gaussian processes under \cite[Hypothesis 5.1]{liu2020discrete}. However, again, this finding does not directly lead to the convergence for every $t$ as required by Definition \ref{def:scaled variation}.

We begin by collecting some basic properties of the scaled quadratic variation in the form of a proposition.

\begin{proposition}\label{prop:sv_properties}
Let $\Pi= \{\pi_n\}$ denote a sequence of partitions of the interval $[0,T]$ with vanishing mesh.  Then 
\begin{enumerate}[(i)]
%\item $\mathcal{V}_{\Pi}^{\gamma}([0,T];\bbR^d)$ is a vector space;
\item $x\in \mathcal{V}_{\Pi}^{\gamma}([0,T])$  if and only if the sequence of measures  
\begin{equation*}
\mu_{n} := \sum_{ [t_i,t_{i+1}] \in \pi_n} (t_{i+1} - t_i)^{\gamma}( \delta x_{t_i t_{i+1}})^2 \delta_{t_i}
\end{equation*}
converges weakly to the measure $\mu$ induced by a bounded variation function;
\item if  $x\in \mathcal{V}_{\Pi}^{\gamma}([0,T])$, then $\langle x \rangle^{(\gamma),n}$ converges to $\langle x \rangle^{(\gamma)}$ uniformly in $t$;
\item  if $(x,z) \in \mathcal{V}_{\Pi}^{\gamma}([0,T];\R^2)$, then for all $\phi \in C([0,T])$,
$$
\lim_{n\rightarrow \infty}\int_0^T \phi_r\rmd  \langle x,  z\rangle^{(\gamma),n}_r=\int_0^T \phi_r \rmd \langle x,  z\rangle^{(\gamma)}_r;
$$
\item given a path $z:[0,T]\rightarrow \bbR$, there is only one non-trivial $\gamma$-scaled variation along $\Pi$; that is, if $0 < \langle z \rangle^{(\gamma)}_T< \infty$, then 
$$
\langle z \rangle^{\beta}_T = \left\{
\begin{array}{ll}
0 & \textrm{ if } \beta > \gamma \\
\infty & \textrm{ if } \beta < \gamma. \\
\end{array}
\right.
$$
\end{enumerate}
\end{proposition}
\begin{proof}
%i) This item follows immediately from Young's inequality. 
(i-iii) We first note that 
$\langle x\rangle_t^{(\gamma),n} \in \textnormal{BV}([0,T])$
is the cumulative distribution function of $\mu_n$. Then we observe that weak convergence of probability measures is equivalent to the uniform convergence of cumulative distribution functions. iv)  If $\beta > \gamma$, then 
\begin{equation} \label{non-trivial variation}
\sum_{ [t_i,t_{i+1}] \in \pi_n}  (t_{i+1} - t_i)^{\beta}  (\delta z_{t_i t_{i+1}})^2 \leq |\pi_n|^{\beta - \gamma}\sum_{ [t_i,t_{i+1}] \in \pi_n} (t_{i+1} - t_i)^{\gamma}  (\delta z_{t_i t_{i+1}})^2 \rightarrow 0
\end{equation}
as $n\rightarrow \infty$. If $\beta < \gamma$, then assuming that $\langle z \rangle^{\beta}_T < \infty$,  we get a contradiction to the assumption that $\langle z \rangle^{(\gamma)}_T>0$ by the same argument as in \eqref{non-trivial variation} with $\gamma$ and $\beta$  interchanged. 
\end{proof}

\subsection{Scaled quadratic variation of controlled paths}

The aim of this section is to show that the properties of the classical quadratic variation from It\^{o}-theory generalizes to scaled variation. In particular, we show that finite $\gamma$-scaled quadratic variation is inherited by paths controlled by paths with finite $\gamma$-scaled covariation. Moreover, we establish a bound that relates the convergence rates of the scaled quadratic variation of a controlled path to convergence rates of the controlling paths covariation in the uniform norm. As a corollary, we establish a  formula for the scaled quadratic variation of solutions of rough differential equations.

We will first prove an auxiliary lemma that relates the convergence rates of the integral of $\phi \in C^{\alpha}([0,T])$ with respect to $\langle x,z\rangle^{(\gamma),n}$ to the convergence rates of $\langle x,z\rangle^{(\gamma),n}$ in the uniform norm. To prove this lemma, we require the following special case of the $(p,q)$-variation Young's sewing lemma \cite[Proposition 2.4]{friz2018differential}. The lemma is a consequence of \cite[Proposition 2.4]{friz2018differential} and a basic interpolation estimate.
\begin{lemma}\label{lem:sewing}
 Let $\phi\in C^{\alpha}([0,T];\bbR)$ and $V\in \textnormal{BV}([0,T];\bbR)$. Then for all $\alpha,\beta> 0$ such that $\alpha+\beta>1 $, there is a constant $C_{\alpha,\beta}>0$ such that 
$$
\left|\int_0^T \phi_r \rmd V_r - \phi_0 \delta V_{0T}\right| \le C_{\alpha,\beta} [\phi]_{\alpha}[V]_{\textnormal{BV}}^{\beta}\|V\|_{\infty}^{1-\beta}\,.
$$
\end{lemma}

\begin{lemma}\label{lem:sewing_bound_basic}
Let $\Pi=\{\pi_n\}$ denote a sequence of partitions of  $[0,T]$ with vanishing mesh. Let $(x,z) \in \mathcal{V}_{\Pi}^{\gamma}([0,T];\bbR^2)$. Then, for all $\alpha \in (0,1]$, $\phi \in C^{\alpha}([0,T])$ and $\epsilon > 0$,
$$
\left|\int_0^T \phi_r \rmd \langle x,z \rangle^{(\gamma),n}_r - \int_0^T \phi_r \rmd \langle x,z \rangle^{(\gamma)}_r \right| =\mathcal{O}_C\left(\|\langle x ,z\rangle^{(\gamma),n} - \langle x ,z\rangle^{(\gamma)}\|_{\infty}^{\alpha-\epsilon} \right)\,,
$$
where $C=C(\alpha, \|\phi\|_{C^\alpha}, \langle x\rangle_{T}^{(\gamma)}, \langle z\rangle_{T}^{(\gamma)})$.
\end{lemma}
\begin{proof}
Let  $\xi^n = \langle x, z\rangle^{(\gamma),n} - \langle x, z\rangle^{(\gamma)}$. Adding and subtracting, we find
$$
\left| \int_0^T \phi_r \rmd \xi^n_r \right| \leq \left| \int_0^T \phi_r \rmd \xi^n_r  - \phi_0 \xi_{0T}^n \right| + |\phi_0 \xi_{0T}^n| .
$$
The latter term is clearly bounded by $2\|\phi\|_{\infty} \|\xi^n\|_{\infty}$. In order to bound the first term, we use Lemma \ref{lem:sewing} to get the existence of a constant $C_{\alpha,\epsilon}$ such that 
$$
\left| \int_0^T \phi_r \rmd \xi^n_r  - \phi_0 \xi_{0T}^n \right| \leq [\phi]_{\alpha} [\xi^n]_{\text{BV}}^{1 - \alpha + \epsilon} \|\xi^n\|_{\infty}^{\alpha - \epsilon}.
$$
By virtue of Young's inequality and the fact that $t\mapsto \langle x\rangle^{(\gamma), n}_t$ and $t\mapsto \langle z\rangle^{(\gamma), n}_t$ are increasing in $t$ and for all $(s,t)\in \Delta_T$,
$$
\|\langle x, z\rangle^{(\gamma), n}\|_{\textnormal{BV},[s,t]} \le \frac{1}{2}\left(\delta \langle x\rangle^{(\gamma), n}_{st} +\delta \langle z\rangle^{(\gamma), n}_{st} \right)
$$
and
$$
\|\langle x, z\rangle^{(\gamma)}\|_{\textnormal{BV},[s,t]} \le \frac{1}{2}\left(\delta \langle x\rangle^{(\gamma)}_{st} +\delta \langle z\rangle^{(\gamma)}_{st} \right)\,.
$$
Thus, 
$$
[\xi^n]_{\text{BV}} \le \frac{1}{2}\left(\langle x \rangle^{(\gamma),n}_T +\langle z \rangle^{(\gamma),n}_T+ \langle x\rangle^{(\gamma)}_T +\langle z\rangle^{(\gamma)}_T\right)\,.
$$
Since $\langle x\rangle^{(\gamma),n}_{T}$ and $\langle z\rangle^{(\gamma),n}_{T}$ converge to $\langle x\rangle^{(\gamma)}_{T}$ and $\langle z\rangle^{(\gamma)}_{T}$, respectively, there is a constant $C>0$ independent of $n$ such that 
$$
[\xi^n]_{\text{BV}}\le C\left(1 + \langle x\rangle^{(\gamma)}_T +\langle z\rangle^{(\gamma)}_T\right)\,.
$$
The claim follows from Proposition \ref{prop:sv_properties} (iv).
\end{proof}

The following result establishes that controlled paths inherit $\gamma$-scaled variation.

\begin{theorem} \label{thm:control_consistency}
Let $\Pi=\{\pi_n\}$ denote a sequence of partitions of  $[0,T]$ with vanishing mesh.
Given $\alpha \in (\frac13,1]$ and $\gamma>0$ satisfying $\gamma+3\alpha >1$, let $z\in C^{\alpha}([0,T];\bbR^K)\cap \mathcal{V}^{\gamma}_{\Pi}([0,T];\bbR^K)$. If $y \in \mathscr{D}^{2 \alpha}_z([0,T])$, then $y\in \mathcal{V}^{\gamma}_{\Pi}([0,T];\bbR)$ and for all $t\in [0,T]$,
\begin{equation}\label{eq:sv_control_paths}
\langle y\rangle^{(\gamma)}_t = \sum_{k,l=1}^K \int_0^t y_s^k y_s^l \rmd\langle z^k, z^l \rangle^{(\gamma)}_r\,,\end{equation}
where $y'=(y^1,\ldots, y^K)$ is the Gubinelli derivative of $y$.
Moreover, for all $t\in [0,T]$ and  $\epsilon>0$,
\begin{gather*}
\langle y\rangle^{(\gamma),n}_t - \sum_{k,l=1}^K \int_0^t y_r^k y_r^l \rmd\langle z^k, z^l \rangle^{(\gamma)}_r = \mathcal{O}_C\left(|\pi_n|^{\gamma + 3\alpha-1} + \max_{1\le k,l\le K}\|\langle z^k, z^l\rangle^{(\gamma),n} - \langle z^k,z^l\rangle^{(\gamma)}\|_{\infty, [0,t]}^{\alpha-\epsilon} \right)\,,
\end{gather*}
where $C=C(\alpha,K, T, \|z\|_{\alpha},\langle \langle z \rangle \rangle_T^{(\gamma)} ,\|y'\|_{\alpha}, [y^{\sharp}]_{2\alpha})$.
\end{theorem}
\begin{proof}
We will only prove convergence rates since \eqref{eq:sv_control_paths} will then follow by Proposition \ref{prop:sv_properties} (iii).  Expanding the square, we find that for all $n\in \bbN$,
\begin{align*}
\langle y\rangle^{(\gamma),n}_t&= \sum_{[t_i, t_{i+1}] \in \pi_n\cap [0,t]} (t_{i+1} - t_i)^{\gamma} (\delta y_{t_i t_{i+1}})^2 \\
&= \sum_{[t_i, t_{i+1}] \in  \pi_n\cap [0,t]} (t_{i+1} - t_i)^{\gamma} \left(\sum_{k=1}^K y_{t_i}^k \delta z^k_{t_{i} t_{i+1}}+ y_{t_i t_{i+1}}^{\sharp} \right)^2 \\
&= \sum_{k,l=1}^K \sum_{[t_i, t_{i+1}] \in \pi_n} (t_{i+1} - t_i)^{\gamma}  y^k_{t_i}y^l_{t_i}  \delta z^k_{t_{i} t_{i+1}} \delta z^l_{t_{i} t_{i+1}}(:=I_1^n)\\
&\qquad + 2\sum_{k=1}^K \sum_{[t_i, t_{i+1}] \in  \pi_n\cap [0,t]} (t_{i+1} - t_i)^{\gamma}  y_{t_i t_{i+1}}^{\sharp} y^k_{t_i} \delta z^k_{t_{i} t_{i+1}}(:=I_2^n) \\
& \qquad + \sum_{[t_i, t_{i+1}] \in \pi_n\cap [0,t]} (t_{i+1} - t_i)^{\gamma} (y_{t_i t_{i+1}}^{\sharp})^2(:=I_3^n)\,.
\end{align*}
We immediately get
\begin{align*}
I_2^n &\le 2\| y'\|_{\infty} [y^{\sharp}]_{2 \alpha} [z]_{\alpha} \sum_{[t_i, t_{i+1}] \in  \pi_n\cap [0,t]} (t_{i+1} - t_i)^{\gamma + 3 \alpha }  \leq 2\| y'\|_{\infty} [y^{\sharp}]_{2 \alpha} [z]_{\alpha} |\pi_n|^{\gamma + 3 \alpha -1} T  \,,\\
I_3^n&\le [y^{\sharp}]_{2 \alpha}^2  \sum_{[t_i, t_{i+1}] \in  \pi_n\cap [0,t]} (t_{i+1} - t_i)^{\gamma + 4 \alpha} \leq [y^{\sharp}]_{2 \alpha}^2 |\pi_n|^{\gamma + 4 \alpha -1} T \,.
\end{align*}
By Lemma \ref{lem:sewing_bound_basic},
$$
I_1 -  \sum_{k,l=1}^K \int_0^t y_r^k y_r^l \rmd\langle z^k, z^l \rangle^{(\gamma)}_r =
\mathcal{O}_{C}\left(\max_{1\le k,l\le K}\|\langle z^k ,z^l\rangle^{(\gamma),n} - \langle z^k,z^l\rangle^{(\gamma)}\|_{\infty}^{\alpha-\epsilon}\right)\,,
$$
where $C=C(\alpha, K, \|y'\|_{\alpha}, \langle \langle z\rangle \rangle_T^{(\gamma)})$. This completes the proof.
\end{proof}

\begin{remark}
Note that \eqref{eq:sv_control_paths} holds even if the Gubinelli derivative is not unique.
\end{remark}

Combining Theorem \ref{thm:control_consistency}, Definition \ref{def:rde},  and Proposition \ref{prop:chain_rule}, we obtain the following corollary, which establishes a formula for the $\gamma$-scaled variation of solutions of rough differential equations.

\begin{corollary}\label{cor:rde_consistency}
Let $\Pi=\{\pi_n\}$ denote a sequence of partitions of  $[0,T]$ with vanishing mesh.
Given $\alpha \in (\frac13,1]$ and $\gamma>0$ satisfying $\gamma+3\alpha >1$, let $\bZ\in \mathscr{C}^{\alpha}([0,T];\bbR^K)\cap \mathcal{V}^{\gamma}_{\Pi}([0,T];\bbR^K)$\footnote{If $\alpha \in (\frac13,\frac12]$, we mean, by abuse of notation, that $\bZ = (z,\mathbb{Z}) \in \mathscr{C}^{\alpha}([0,T];\bbR^K)$ where $z \in \mathcal{V}^{\gamma}_{\Pi}([0,T];\bbR^K)$}.  Let $y$ be a solution of 
$$
\rmd y_t = u_t(y_t)\rmd t + \sum_{k=1}^K \sigma_k(y_t) \rmd \bZ^k
$$
as in Definition \ref{def:rde}. Then for all $f \in C^2(\bbR^d)$ and $t\in [0,T]$,
$$
\langle f(y)\rangle_t^{(\gamma)} = \sum_{k,l=1}^K \int_0^t \sigma_k[f] (y_r)\sigma_l[f] (y_r)  \rmd \langle z^k, z^l \rangle^{(\gamma)}_r \,.
$$
Moreover,  for all $t\in [0,T]$ and $\epsilon>0$, 
\begin{gather*}
\langle f(y)\rangle^{(\gamma),n}_t - \sum_{k,l=1}^K \int_0^t \sigma_k[f] (y_r)\sigma_l[f] (y_r)  \rmd \langle z^k, z^l \rangle^{(\gamma)}_r  \\= \mathcal{O}_C\left(|\pi_n|^{\gamma + 3\alpha-1} + \max_{1\le k,l\le K}\|\langle z^k, z^l\rangle^{(\gamma),n} - \langle z^k,z^l\rangle^{(\gamma)}\|_{\infty, [0,t]}^{\alpha-\epsilon} \right)\,,
\end{gather*}
where $C=C(T, \|z\|_{\alpha},\langle  \langle z\rangle\rangle_T^{(\gamma)} ,\|f(y)'\|_{\alpha},[f(y)^{\sharp}]_{2\alpha})$.
\end{corollary}

\subsection{Estimation of the scaling exponent}
In this section, we introduce an estimator of the scaling exponent $\gamma$ and relate the convergence rates of the estimator to the convergence rates of the scaled quadratic variation of path. The estimator is based on subsampling the (discrete) observation of a path  $x \in \mathcal{V}^{\gamma}_{\Pi}([0,T];\bbR^K)$ at equidistant times. 

\begin{theorem}\label{thm:estimating_gamma}
Let $\Pi=\{\pi_n\}$ denote a uniform sequence of partitions of $[0,T]$ with vanishing mesh. Let $x \in \mathcal{V}^{\gamma}_{\Pi}([0,T];\bbR^K)$. Let $\lambda = \{\lambda_n\}\subset \bbN$ be such that  the limit of $\Delta_{\lambda , n}:= \frac{|\pi_{\lambda_n}|}{|\pi_n|}$ exists and is not equal to one. Denote the limit by $\Delta_{\lambda}$. Define for all $n\in \bbN$,
$$
\gamma_{\lambda,n} = -\frac{\ln S_{\lambda,n}}{\ln \Delta_{\lambda ,n}}\,,   \qquad
\textrm{where}
\qquad  S_{\lambda,n}= \frac{ \langle \langle x \rangle \rangle_T^{(0),\lambda_n}}{  \langle \langle x \rangle \rangle_T^{(0),n}} 
\,.
$$
Then, for all $n\in \bbN$,
\begin{align*}
\left|\gamma - \gamma_{\lambda, n}\right| 
& \le \frac{2}{|\ln \Delta_{\lambda, n}|} \max_{\tau \in \{n,\lambda_n\}}\frac{| \langle \langle x 
\rangle\rangle^{(\gamma) ,\tau }_T -  \langle \langle x \rangle \rangle^{(\gamma)}_T|}{ 
\langle \langle x  \rangle\rangle^{(\gamma) ,\tau }_T \wedge  \langle\langle x \rangle\rangle^{(\gamma)}_T}\,.
\end{align*}
In particular,  $\gamma_{\lambda, n}$ is a consistent estimator of $\gamma$ and 
$$
\gamma - \gamma_{\lambda, n} = \mathcal{O}_{C}\left( \langle\langle  x\rangle\rangle^{(\gamma)}_T - \langle \langle x\rangle \rangle^{(\gamma) , n}_T\right)\,,
$$
where $C=C(\langle\langle x \rangle\rangle^{(\gamma)}_T, \Delta_{\lambda})$.
\end{theorem}
\begin{proof}
We prove the statement for $K=1$. The general result follows easily by the same considerations.  Since $\pi$ is uniform, then for all $n\in \bbN$ and $\tau \in \{n,\lambda_n\}$,
$$
\langle x \rangle_T^{(\gamma),\tau} = \sum_{[u,v] \in \pi_{\tau}} (v-u)^{\gamma} |\delta x_{uv}|^2 = |\pi_n|^{\gamma}\sum_{[u,v] \in \pi_{\tau}}  |\delta x_{uv}|^2 = |\pi_{\tau}|^{\gamma} \langle x \rangle_T^{(0),{\tau}}
$$
which implies
$$
\frac{\langle x\rangle^{(\gamma) ,\lambda_n}_T }{\langle x \rangle^{(\gamma) ,n}_T}=\frac{|\pi_{\lambda_n}|^{\gamma} \langle x\rangle_T^{(0),\lambda_n}}{|\pi_{n}|^{\gamma} \langle x\rangle_T^{(0),n} }= \Delta_{ \lambda, n}^{\gamma} S_{\lambda,n}\,.
$$
Thus,
$$
\gamma- \gamma_{\lambda_n}=\frac{1}{\ln \Delta_{\lambda, n}} \ln \left(\frac{\langle x\rangle^{(\gamma) ,\lambda_n}_T }{\langle x \rangle^{(\gamma) ,n}_T}\right) = \frac{1}{\ln \Delta_{\lambda, n}} \left( \ln \langle x\rangle^{(\gamma) ,\lambda_n}_T  - \ln \langle x \rangle^{(\gamma) ,n}_T \right)\,.
$$
Adding and subtracting $\ln \langle x \rangle^{(\gamma)}_T$, we find
\begin{align*}
\left|\gamma - \gamma_{\lambda, n }\right| & \le  \frac{1}{|\ln \Delta_{\lambda, n}|} \left(|\ln \langle x\rangle^{(\gamma) ,\lambda_n}_T -  \ln \langle x\rangle^{(\gamma)}_T| + |\ln \langle x\rangle^{(\gamma) ,n}_T -  \ln \langle x\rangle^{(\gamma)}_T| \right) \\
& \le \frac{1}{|\ln \Delta_{\lambda, n}|} \left( \frac{|\langle x\rangle^{(\gamma) ,\lambda_n}_T -  \langle x\rangle^{(\gamma)}_T|}{\langle x\rangle^{(\gamma) ,\lambda_n}_T \wedge \langle x\rangle^{(\gamma)}_T}+  \frac{|\langle x\rangle^{(\gamma) ,n}_T -  \langle x\rangle^{(\gamma)}_T|}{\langle x\rangle^{(\gamma) ,n}_T \wedge \langle x\rangle^{(\gamma)}_T} \right)\,,
\end{align*}
where in the second inequality, we have used the mean value theorem $|\ln(a) - \ln(b) |\leq \frac{|a-b|}{a \wedge b}$.  Since $$\lim_{n\rightarrow \infty} \langle x\rangle^{(\gamma) ,n}_T=\lim_{n\rightarrow \infty}\langle x\rangle^{(\gamma) ,\lambda_n}_T=\langle x\rangle^{(\gamma)}_T\,,$$ there is an $N>0$ such that for all $n\ge N$
$$
\langle x\rangle^{(\gamma) ,n}_T>\langle x\rangle^{(\gamma)}_T/2  \quad \textnormal{and} \quad \langle x\rangle^{(\gamma) ,\lambda_n}_T>\langle x\rangle^{(\gamma)}_T/2\,.
$$
The result follows. 
\end{proof}
\begin{remark}\label{rem:Delta of 2 partitions}
For the sequence of partitions $\pi_n = \{ \frac{iT}{n}\}_{i=0}^n$, a natural choice of $\lambda$ is $\lambda_n = 2n$, and for the dyadic sequence of partitions $\pi_n = \{ \frac{iT}{2^n} \}_{i=0}^{2^n}$, a natural choice of $\lambda$ is $\lambda_n = n+1$. In both cases, we have $\Delta_{\lambda, n} = \frac12$. 
\end{remark}

\begin{proposition} \label{prop:double_approximation}
Let the notation and assumptions of Theorem \ref{thm:estimating_gamma} hold. Assume, in addition, that $\Pi$ is such that for each $n\in \bbN$,
\begin{equation} \label{asm:convergence_of_alpha_est}
( \gamma - \gamma_{\lambda,n} ) \ln(|\pi_n|) = \mathcal{O}(1).
\end{equation}
Then, 
$$
\langle \langle x \rangle  \rangle^{(\gamma)}_T  - \langle \langle x \rangle  \rangle^{(\gamma_{\lambda,n}),n}_T   = \mathcal{O}_C\left(  \left(\langle \langle x\rangle \rangle^{(\gamma)}_T - \langle \langle x\rangle \rangle^{(\gamma) , n}_T\right) \ln (|\pi_n|) \right)\,,$$
where $C=C(\langle\langle x \rangle\rangle^{(\gamma)}_T, \Delta_{\lambda})$.
\end{proposition}
\begin{proof}
As in the previous proof, we restrict to $d=1$. Clearly, it is enough to study the difference $\langle x \rangle^{(\gamma_{\lambda,n}),n}_T - \langle x \rangle^{(\gamma),n}_T$.
Towards this end, we write
\begin{align}
\langle x \rangle^{(\gamma_{\lambda,n}),n}_T - \langle x \rangle^{(\gamma),n}_T
&=\sum_{[t_i,t_{i+1}]\in \pi_n} [(t_{i+1}-t_i)^{\gamma_{\lambda,n}} -(t_{i+1}-t_i)^{\gamma}]\delta x_{t_it_{i+1}}^2\notag \\
& = \sum_{[t_i,t_{i+1}]\in \pi_n} (t_{i+1}-t_i)^{\gamma} \delta x_{t_it_{i+1}}^2 \left[ (t_{i+1}-t_i)^{\gamma_{\lambda,n} - \gamma}  - 1 \right].\label{eq:double_estimator_expression}
\end{align}
Applying the  mean value theorem $|e^a - e^b| \leq e^{a \vee b} |a-b|$, assumption \eqref{asm:convergence_of_alpha_est}, and Theorem \ref{thm:estimating_gamma}, we have
\begin{align*}
(t_{i+1}-t_i)^{\gamma_{\lambda,n} - \gamma}  - 1 & = \exp( (\gamma_{\lambda,n} - \gamma) \ln (t_{i+1}-t_i) )- \exp(0) \\
&\leq \exp( (\gamma_{\lambda,n} - \gamma) \ln (t_{i+1}-t_i) \vee 0 )|\gamma_{\lambda,n} - \gamma| |\ln (t_{i+1}-t_i)| \\
& \leq \mathcal{O}_C( (\langle x \rangle^{(\gamma),n}_T - \langle x \rangle^{(\gamma)}_T ) \ln (|\pi_n|)) \,.
\end{align*}
Plugging this into \eqref{eq:double_estimator_expression}, we obtain
$$
\langle x \rangle^{(\gamma_{\lambda,n}),n}_T - \langle x \rangle^{(\gamma),n}_T\le \langle x \rangle^{(\gamma),n}_T \mathcal{O}_C( (\langle x \rangle^{(\gamma),n}_T - \langle x \rangle^{(\gamma)}_T ) \ln (|\pi_n|))\,,
$$
which completes the proof.
\end{proof}

\section{Convergence rates in the fractional Brownian motion setting} \label{sec:fBm_convergence_rates}
In this section, we derive convergence rates for the scaled quadratic variation of fBm and apply it to obtain the convergence rates for the scaled quadratic variation of fractional diffusions.

\begin{definition}\label{def:partitition_with_rate}
A sequence of uniform partitions $\Pi= \{ \pi_n\}$ of the interval $[0,T]$ is said to be admissible if  there exists a function $g:\bbN\rightarrow \bbR$ such that 
$$
\sum_{n=1}^{\infty} \#\pi_n e^{-g(\#\pi_n)}  < \infty .
$$
For such a partition, we define for all $n\in \bbN$,
\begin{equation}\label{def:rate}
\delta_n(\Pi)=
\left\{
\begin{array}{ll}
\sqrt{  T |\pi_n|g(\#\pi_n ) }  & \textrm{ if } H \leq \frac12 \\
\sqrt{T |\pi_n|^{2-2H} g(\#\pi_n )} & \textrm{ if } H > \frac12 \\
\end{array}
\right.\,.
\end{equation}
\end{definition}
\begin{remark}\label{rem:rate_dyadic}
For uniform partitions $\Pi=\{\pi_n\}$ with $\#\pi_n = n$ and $\#\pi_n = 2^n$ (i.e., dyadic), we may choose $g(n) = 3 \ln(n)$ and $g(2^n) = n$, respectively.
\end{remark}

The following proposition establishes a convergence rate for the scaled quadratic variation of fBm in the uniform norm. The proof is inspired by \cite[Proposition 5.2]{han2021hurst}, \cite[Theorem 2.9]{kubilius2017parameter}, \cite[Theorem 2]{klein1975quadratic}.

\begin{proposition} \label{prop:fBm_sv_rates}
Let $\Pi$ and $\delta(\Pi)$ be as in Definition \eqref{def:partitition_with_rate}. Let $B$ denote  a  fractional Brownian motion with Hurst parameter $H \in (0, 1)$. Then, $\bbP$-a.s,   $B \in \mathcal{V}^{1-2H}_{\Pi}([0,T];\R)$ and 
$$
\sup_{t \in [0, T] } \left| \langle B \rangle^{(1 - 2H),n}_{t}  - t \right| = \mathcal{O}(\delta_n(\Pi))\,.
$$
\end{proposition}
\begin{proof}
Since for any $t\in [0,T]$, we have 
$$
t -  \sum_{[t_i,t_{i+1}]\in \pi_n  \cap [0,t]} (t_{i+1} - t_i) = \mathcal{O}(|\pi_n|)\,,
$$
it is enough to show the convergence rate for 
$$
\sup_{t \in [0,T]} \left|\sum_{[t_i,t_{i+1}]\in \pi_n \cap [0,t]} \left( (t_{i+1}-t_i)^{1-2H} \delta B_{t_i t_{i+1}}^2 - (t_{i+1} - t_i)  \right) \right|\,.
$$
Henceforth, we denote $Q_i=(t_{i+1}-t_i)^{1-2H} \delta B_{t_i t_{i+1}}^2 - (t_{i+1} - t_i)$.
Notice that the function 
$$
t\mapsto \left|\sum_{[t_i,t_{i+1}]\in \pi_n \cap [0,t]} Q_i\right|
$$
is constant except at the partition points of $\pi_n$, at which it is potentially discontinuous. Thus, for all $\delta > 0$,
\begin{align}
\mathbb{P}\left( \sup_{t\in [0,T]}\left|\sum_{[t_i,t_{i+1}]\in \pi_n \cap [0,t]} Q_i\right|> \delta\right) &= \mathbb{P}\left( \bigcup_{[t_k,t_{k+1}] \in \pi_n}  \left\{\left|\sum_{[t_i,t_{i+1}]\in \pi_n \cap [0,t_{k+1}]} Q_i \right| > \delta\right\}\right) \notag \\
& \leq \sum_{[t_k,t_{k+1}] \in \pi_n} \mathbb{P}\left( \left|\sum_{[t_i,t_{i+1}]\in \pi_n \cap [0,t_{k+1}]}  Q_i \right| > \delta\right) \label{eq:sum bound}\,.
\end{align}
We will now bound  $\mathbb{P}\left( \left|\sum_{[t_i,t_{i+1}]\in \pi_n \cap [0,t_{k+1}]}  Q_i \right| > \delta\right)$. Towards this end, consider the $N : = \sharp (\pi_n \cap [0,t_{k+1}])$-dimensional Gaussian vector
$$
Z_i := (t_{i+1}-t_i)^{\frac12 -H } \delta  B_{t_i t_{i+1}} , \qquad i = 0, \dots, N-1\,,
$$
and denote by $A = (a_{ij})$ its covariance matrix $a_{ij} = \mathbb{E}(Z_i Z_j)$. Since $A$ is positive semi-definite, consider the diagonalization $A = U D U^T$, where $D = \textrm{diag}(\lambda_1, \dots \lambda_N)$ and $U^TU=I_{N \times N}$. Define $X =  D^{-1/2} U^T Z$, which is Gaussian with independent entries. The Hanson-Wright inequality \cite{vershynin2013} implies that there is a universal constant $c>0$ such that 
\begin{equation} \label{eq:HW}
\mathbb{P} \left( \big| X^T D X - \mathbb{E}( X^T D X)\big| > \delta \right) \leq 2 \exp\left( - c \min\left( \frac{\delta^2}{ 4 \|D\|_2^2} , \frac{\delta}{ 4 \|D\|} \right) \right) ,
\end{equation}
where $\|\cdot \|$ denotes the spectral norm. 
Notice that
$$
\langle B \rangle_{t_{k+1}}^{(1-2H),n} = Z^T Z = X^T D X\,,
$$
and so $\mathbb{E}( X^T D X) = \mathbb{E}( \langle B \rangle_{t_{k+1}}^{(1-2H),n}) = t_{k+1}$. 

Basic considerations show that 
\begin{gather*}
\|D \|_2^2 = \sum_{i=1}^N \lambda_i^2 \leq \|D\| \textrm{trace}(D) = \|D\| \textrm{trace}(A)\,,\\
\textrm{trace}(A) = \sum_i \mathbb{E}((Z_i)^2) = \sum_i (t_{i+1}-t_i)^{1 -2H } \mathbb{E}((\delta B_{t_i t_{i+1}})^2) = t_{k+1} \leq T\,,
\end{gather*}
and $\|D\| = \|A\|$.
Plugging this into \eqref{eq:HW}, we get
\begin{equation} \label{eq:HWfBm}
\mathbb{P} \left( \big| \langle B  \rangle_{t_{k+1}}^{(1-2H),n} - t_{k+1}  \big| > \delta \right) \leq 2 \exp\left( -   \frac{ c \delta^2}{4 T \|A\|} \right) \,.
\end{equation}

It remains to find a suitable bound on the largest eigenvalue of the covariance matrix $A$, for which we will use the Gershgorin's circle theorem \cite{gershgorin1931uber}: for all eigenvalues  $\lambda$ of $A$,
$$
\lambda \leq \max_{0 \leq j \leq N-1} \sum_{i=0}^{N-1} |a_{ij}|\,.
$$
Consider first the case $H \leq \frac12$ and let $j\in \{0,\ldots, N-1\}$. If $j \neq i$, then \eqref{eq:cov_neg_pos} gives
$
\mathbb{E}(\delta B_{t_i t_{i+1}}\delta B_{t_j t_{j+1}}) \leq 0.
$\footnote{If $H=\frac12$, then this is of course equal to 0.}
Thus, 
\begin{align*}
\sum_{i=0}^{N-1} |a_{ij}| & = \sum_{i\neq j} |a_{ij}| + |a_{jj}| = - \sum_{i\neq j} a_{ij} + a_{jj} = -\sum_{i=0}^{N-1} a_{ij} + 2a_{jj} \\
& = - |\pi_n|^{1-2H}\sum_{i=0}^{N-1} \mathbb{E}( \delta B_{t_i t_{i+1}} \delta B_{t_j t_{j+1}}) + 2 |\pi_n|^{1-2H}  \mathbb{E}( (\delta B_{t_j t_{j+1}})^2) \\
& = - |\pi_n|^{1-2H} \mathbb{E}( \delta B_{t_0 t_{k+1}} \delta B_{t_j t_{j+1}}) + 2 |\pi_n|^{1-2H}  \mathbb{E}( (\delta B_{t_j t_{j+1}})^2).
\end{align*}
By \eqref{eq:cov_subset}, we have $\mathbb{E}( \delta B_{t_0 t_{k+1}} \delta B_{t_j t_{j+1}}) \leq |t_{j+1} - t_j|^{2H}$ since $[t_j, t_{j+1}] \subset [t_0, t_{k+1}]$. Clearly, we have $\mathbb{E}( (\delta B_{t_j t_{j+1}})^2) = |t_{j+1} - t_j|^{2H}$, which yields
$
\| A \| \leq 3 |\pi_n| .
$
Consider now the case $H > \frac12$ and let $j\in \{0,\ldots, N-1\}$. If $j \neq i$, then $
\mathbb{E}(\delta B_{t_i t_{i+1}}\delta B_{t_j t_{j+1}}) \geq 0$ from \eqref{eq:cov_neg_pos}
so that 
\begin{align*}
\sum_{i} |a_{ij}| &= \sum_{i} a_{ij} = |\pi_n|^{1-2H} \sum_i \mathbb{E}( \delta B_{t_i t_{i+1}} \delta B_{t_j t_{j+1}}) \\
& = |\pi_n|^{1-2H} \mathbb{E}( \delta B_{t_0 t_{k+1}}  \delta B_{t_j t_{j+1}}) \leq 3 |\pi_n|^{2-2H},
\end{align*}
which gives that 
$$
\| A \| \leq 3 |\pi_n|^{2-2H}.
$$

Combining the estimates above, we have 
$$
\| A \| \leq  \varepsilon_n :=  
\left\{
\begin{array}{ll}
3 |\pi_n| & \textrm{ if } H \leq \frac12 \\ 
3 |\pi_n|^{2 - 2H} & \textrm{ if } H > \frac12 \\ 
\end{array}
\right. \,.
$$
Inserting this bound into \eqref{eq:HWfBm}, find that there is a constant $c>0$, which has changed,  such that 
$$
\mathbb{P} \left( \big|  \langle B  \rangle_{t_{k+1}}^{(1-2H),n} - t_{k+1}  \big| > \delta \right) \leq 2 \exp\left( -   \frac{ c \delta^2}{ T \varepsilon_n} \right)\,.
$$
If we choose 
$$
\hat{\delta}_n = \sqrt{ (3T \varepsilon_n  g (\# \pi_n ))/c}\,,,
$$ 
then
\begin{equation} \label{eq:HWfBmH<0.5delta}
\mathbb{P} \left( \left| \langle B  \rangle_{t_{k+1}}^{(1-2H),n} - t_{k+1} \right| > \delta_n \right) \leq 2e^{-g(\#\pi_n )} .
\end{equation}
Plugging this back into \eqref{eq:sum bound}, we find 
$$
\mathbb{P}\left( \sup_{t \in [0,T]}\left|\sum_{\pi_n \cap [0,t]} \left( (t_{i+1}-t_i)^{1-2H} \delta B_{t_i t_{i+1}}^2 - (t_{i+1} - t_i)  \right) \right| > \delta_n\right) \notag \leq 2 \#\pi_n e^{-g(\#\pi_n )} .
$$
We then apply the Borel-Cantelli lemma to complete the proof since the sequence is summable in $n$ by assumption. 
\end{proof}

\begin{remark}
In \cite{han2021hurst}, the authors prove convergence rates at a fixed time $t = 1$ if the sequence of partitions is the dyadic, i.e., $\#\pi_n = 2^n$.
By modifying the above proof, we can improve the rates for $H= \frac12$ from \cite{han2021hurst} as follows. In \eqref{eq:HWfBmH<0.5delta}, we may choose $t_{k+1}= 1$ and let $g$ be such that $g(2^n) = 2 \ln(n)$. Then the right-hand-side of \eqref{eq:HWfBmH<0.5delta} is summable in $n$, which gives the rates
\begin{equation*}
| \langle B \rangle^{(1 - 2H),n}_{t}  - t | =
\left\{
\begin{array}{ll}
\mathcal{O}( 2^{-n/2} \sqrt{\ln(n) } ) & \textrm{ if } H \leq \frac12 \\
\mathcal{O}( 2^{-n(1-H)}  \sqrt{ \ln(n)} ) & \textrm{ if } H > \frac12 \\
\end{array}
\right.
, \qquad \mathbb{P}-a.s. .
\end{equation*}
Since our framework hinges on convergence uniformly in $t$, we get a slightly worse rate
$$
\sup_{t \in [0,T]} | \langle B \rangle^{(1 - 2H),n}_{t}  - t | =
\left\{
\begin{array}{ll}
\mathcal{O}( 2^{-n/2} \sqrt{n } ) & \textrm{ if } H \leq \frac12 \\
\mathcal{O}( 2^{-n(1-H)}  \sqrt{ n} ) & \textrm{ if } H > \frac12 \\
\end{array}
\right.
, \qquad \mathbb{P}-a.s. .
$$
\end{remark}

The following proposition, whose proof is essentially the same as the above, establishes the same convergence rates the scaled \emph{covariation} of two independent fBms. 
\begin{proposition} \label{prop:fBm_cv_rates}
Let $\Pi$ and $\delta(\Pi)$ be as in Definition \eqref{def:partitition_with_rate} and $B$ denote  a $K$-dimensional fBm with Hurst parameter $H \in (0, 1)$.
Then, $\bbP$-a.s., 
$$
\sup_{t \in [0,T]}| \langle B^k,B^l \rangle^{(1 - 2H),n}_{t} - \delta_{k,l} t|= \mathcal{O}(\delta_n(\Pi))\,.
$$
\end{proposition}
\begin{proof}
If $k=l$, then this is Proposition \ref{prop:fBm_sv_rates}. If $k \neq l$, then the polarization identity yields 
$$
\langle B^k, B^l \rangle^{(1 - 2H),n}_{t}= \frac12\left(\left( \left\langle \frac{B^k+ B^l}{\sqrt{2}} \right\rangle^{(1-2H),n}_{t}  -t\right) +  \left( \left\langle \frac{B^k- B^l}{\sqrt{2}} \right\rangle^{(1-2H),n}_{t} - t \right) \right).
$$
Using \eqref{eq:sum_of_fBm} (i.e., $\frac{B^k \pm B^l}{\sqrt{2}}  \sim B^1$) and that the proof of Proposition \ref{prop:fBm_sv_rates} depends only on the distribution of fBm, the result follows.
\end{proof}

Combining Corollary \ref{cor:rde_consistency} and Proposition \ref{prop:fBm_cv_rates}, we obtain the following. 

\begin{theorem} \label{thm:rate_for_fractional_diffusion}
Let $\Pi=\{\pi_n\}$ denote a sequence of partitions of  $[0,T]$ with vanishing mesh. 
Let $B=(B^1,\ldots,B^K)$ denote an $\bbR^K$-valued fractional Brownian motion with Hurst index $H\in (1/3, 1)$ and  $\bB\in \mathscr{C}^{H-\epsilon}([0,T];\bbR^K)$, $\epsilon>0$, denote its rough path lift as in Theorem \ref{thm:fBm_lift}. Let $y$ denote a pathwise solution of 
$$
\rmd y_t = u_t(y_t)\rmd t + \sum_{k=1}^K \sigma_k(y_t) \rmd \bB^k_t\,.
$$ 
Then, for all $f\in C^2(\bbR^d)$, $\bbP$-a.s.\ for all $t\in [0,T]$, 
$$
\langle f(y) \rangle^{(1-2H)}_t =  \sum_{k=1}^K \int_0^t  |\sigma_k[f](y_r)|^2 \rmd r\,.
$$
Moreover, if $\Pi$ and $\delta(\Pi)$ are as in Definition \eqref{def:partitition_with_rate}, then for all  $f\in C^2(\bbR^d)$ and $\epsilon>0$, $\bbP$-a.s., 
$$
\langle f(y) \rangle^{(1-2H),n}_t-\sum_{k=1}^K \int_0^t |\sigma_k[f](y_r)|^2 \rmd r = \mathcal{O}(\delta_n(\Pi)^{H-\epsilon}) \,,
$$
where $C=C(T, \|B\|_{H-\epsilon},\|f(y)'\|_{H-\epsilon},[f(y)^{\sharp}]_{2(H-\epsilon)})$.
\end{theorem}
\begin{proof}
Let $\epsilon > 0$ be given and let $\tilde{\Omega}$ be the set of all $\omega \in \Omega$ such that $\bB(\omega) \in \mathscr{C}^{H-\epsilon}([0,T];\bbR^K)$ and the rate of Proposition \ref{prop:fBm_cv_rates} holds. It is clear that $\tilde{\Omega}$ has full measure. The result now follows from Corollary \ref{cor:rde_consistency} and the observation that 
$$
\mathcal{O}(|\pi_n|^{1 - 2H + 3(H-\epsilon)-1} + \delta_n(\Pi)^{H-\epsilon}) = \mathcal{O}(\delta_n(\Pi)^{H-\epsilon}).
$$
\end{proof}

\section{Estimating parameters for fractional diffusions}

The goal of this section is to use scaled quadratic variation for parameter estimation in rough differential equations. We focus on fBm as the driving noise, but it is clear that the method works with minor modifications for any driving path with known scaled variation and rates of convergence analogous to Propositions \ref{prop:fBm_sv_rates} and \ref{prop:fBm_cv_rates}. 

Throughout the section, we let $B=(B^1,\ldots,B^K)$ denote an $\bbR^K$-valued fractional Brownian motion with Hurst index $H\in (1/3, 1)$ and  $\bB\in \mathscr{C}^{H-\epsilon}([0,T];\bbR^K)$, $\epsilon>0$, denote its rough path lift as in Theorem \ref{thm:fBm_lift}. Given a set of parameterized vector fields  $\sigma_k : \R^d \times \Theta \rightarrow\R^d$, $k\in \{1,\ldots, K\}$, we assume we observe a pathwise solution of
\begin{equation} \label{eq:RDE_fBm_theta}
\rmd y_t = u_t(y_t) \rmd t + \sum_{k=1}^K \sigma_k(y_t;\theta) \rmd \bB_t^k\,.
\end{equation}
along a sequence of dyadic partitions $\Pi=\{\pi_n\}$ of the interval $[0,T]$ defined by $\pi_n = \{ iT2^{-n} \}_{i=0}^{2^n}$ for all $n
\in \bbN$. For notational simplicity (see Definition \ref{def:rate} and Remark \ref{rem:rate_dyadic}), we define for all $n\in \bbN$,
$$
\delta_n = \delta_n(\Pi)=
\left\{
\begin{array}{ll}
2^{-n/2} \sqrt{n }  & \textrm{ if } H \leq \frac12 \\
2^{-n(1-H)}  \sqrt{ n}  & \textrm{ if } H > \frac12 \\
\end{array}
\right.\,.
$$

\subsection{Estimating the Hurst index}

We begin by finding an estimator of the Hurst index $H$.

\begin{theorem} \label{thm:estimating_H}
Let $\Pi=\{\pi_n\}$ denote a sequence of dyadic partitions of  $[0,T]$ with vanishing mesh and let $y$ denote a pathwise solution of \eqref{eq:RDE_fBm_theta}. Define for all $n\in \bbN$, 
\begin{equation}\label{def:hurst_estimator}
H_n  =\frac{1}{2}\left(1 -  \frac{ \ln S_{n}}{\ln 2 }\right)\,, \qquad \textrm{where} \qquad 
S_{n} = \frac{\langle\langle y \rangle\rangle_T^{0, n+1}}{\langle\langle y\rangle\rangle_T^{0,n}}\,.
\end{equation}
Then, for every $\epsilon >0$, $\mathbb{P}$-a.s.,
$$
H - H_{n} = \mathcal{O}_{C}\left( \delta_n^{H- \epsilon} \right),
$$
where $C=C(T, \|B\|_{H-\epsilon} ,\|y'\|_{H-\epsilon},[y^{\sharp}]_{2(H-\epsilon)}).$
\end{theorem}
\begin{proof}
By Theorem \ref{thm:estimating_gamma} and Remark \ref{rem:Delta of 2 partitions} (i.e., $\Delta_{\lambda,n} = \frac12$ if we let $\lambda_n = n+1$), we have
$$
H - H_{n} =  \mathcal{O}
\left(\langle\langle y \rangle^{(1-2H), n}_T - \langle \langle y \rangle\rangle _t^{(1-2H)} \right)\, .
$$
Applying Theorem \ref{thm:rate_for_fractional_diffusion} with coordinate functions, we complete the proof. 
\end{proof}

Owing to Theorem \ref{thm:estimating_H}, we have that $(H - H_{n}) \ln (|\pi_n|) = \mathcal{O}(n \delta_n^{H-\epsilon})$. Thus, from Theorem \ref{thm:rate_for_fractional_diffusion} and Proposition \ref{prop:double_approximation}, we immediately obtain the following.

\begin{corollary}\label{cor:double_approx_rates}
Let $\Pi=\{\pi_n\}$ denote a sequence of dyadic partitions of  $[0,T]$ with vanishing mesh and let $y$ denote a pathwise solution of \eqref{eq:RDE_fBm_theta}. Then, then for all  $f\in C^2(\bbR^d)$ and $\epsilon>0$, $\bbP$-a.s.\, for all $t\in [0,T]$,
$$
\langle f(y) \rangle^{(1-2H_{n}),n}_t-\sum_{k=1}^K \int_0^t |\sigma_k[f](y_r;\theta)|^2 \rmd r   = \mathcal{O}\left( n \delta_n^{H-\epsilon}\right)\,,
$$
where $C=C(T, \|B\|_{H-\epsilon},\|f(y)'\|_{H-\epsilon},[f(y)^{\sharp}]_{2(H-\epsilon)}).$
\end{corollary}

\subsection{Estimating parameters of a linear parameterized vector fields} 

By Theorem \ref{thm:rate_for_fractional_diffusion}, letting $y$ be a pathwise solution of \eqref{eq:RDE_fBm_theta}, for all $f\in C^2(\bbR^d)$, $\bbP$-a.s., 
$$
\langle f(y) \rangle^{(1-2H)}_t = \sum_{k=1}^K \int_0^t \left(\sigma_k[f](y_r, \theta)\right)^{2} \rmd r\,,
$$
which motivates estimating $\theta\in \Theta$ from single trajectory $y$ by minimizing the squared error
\begin{equation} \label{eq:squared error}
L(\theta) := \frac{1}{2} \sum_{f \in \mathbb{F}} \left| \langle f(y) \rangle^{(1-2H)}_T - \sum_{k=1}^K \int_0^T |\sigma_k[f](y_r; \theta)|^{2} \rmd r \right|^2 \,,
\end{equation}
where  $\mathbb{F}=\{f_1, \dots, f_M\}\subset C^2(\R^d)$ is sufficiently rich. In our analysis, we will focus on the case of linear parameterized vector fields. That is, we assume we observe the solution $y$ of
\begin{equation}\label{eq:RDE_fBm_theta_linear}
\rmd y_t = u_t(y_t) \rmd t +  \sum_{k=1}^K \theta_{k} \sigma_k(y_t) \rmd \bB_t^k\,,
\end{equation}
along a dyadic partition $\pi_n = \{ iT2^{-n} \}_{i=0}^{2^n}$, where $\theta = (\theta_1, \dots, \theta_K)\in \bbR^K_+$. By Theorem \ref{thm:rate_for_fractional_diffusion}, $\bbP$-a.s., for all $t\in [0,T]$,
\begin{equation} \label{eq:exact_equation}
\langle f(y) \rangle^{(1-2H)}_t = \sum_{k=1}^K \theta_k^2 \int_0^t \left(\sigma_k[f](y_r)\right)^{2} \rmd r\,.
\end{equation}

For $m\in \{1,\ldots, M\}$ and $k\in \{1,\ldots, K\}$, define
\begin{equation} \label{eq:def least square matrices}
Y^m=\langle f_m(y)\rangle_{T}^{(1-2H)}, \quad 
X^{mk}=\int_0^T \sigma_k[f_m]^2(y_r) \rmd r\,, \quad \textnormal{and} \quad  \beta_k = \theta_k^2\,.
\end{equation}
Let  $Y=(Y^m)_{1 \leq m \leq M}$, $X=(X^{m,k})_{1 \leq k \leq K, 1 \leq m \leq M}$, and $\beta = (\beta_k)_{1 \leq k \leq K}$. By \eqref{eq:exact_equation}, we have 
$
Y = X \beta.
$
Reparameterizing \eqref{eq:squared error} with the $\beta=\theta^2$, we introduce the quadratic loss 
$$
L(\beta) = \frac{1}{2} |Y - X \beta|^2 .
$$
If the columns of $X$ are linearly independent, then it is well known that $L$ is minimized at 
$$
\beta = (X^T X)^{-1} X^T Y
$$

To obtain rates of convergence, we will use the following quantitative stability result for the least-square  problem from \cite[Theorem 5.3.1]{golub2013matrix}. 

\begin{theorem}\label{thm:least_square_stability}
Let $Y, \{Y_n\}\in \R^{M}$
and $X, \{X_n\}\in \R^{M\times K}$. Let 
$$
\beta=\operatorname{arg\,min}_{\beta} |Y - X \beta| \quad \textnormal{and} \quad \beta_n=\operatorname{arg\,min}_{\beta} |Y_n - X_n \beta|\,.
$$
Assume that  $X_n \rightarrow X$ and $Y_n \rightarrow Y$ as $n\rightarrow \infty$, the columns of $X$ are linearly independent, and
\begin{equation}\label{asm:general_stability_ne}
|X \beta - Y| \neq |Y|\,.
\end{equation}
Then,
$$
|\beta - \beta_n | = \kappa(X)^2 \mathcal{O}_{\|X\|_2 \vee |Y|}(\|X - X_n\|_2 \vee |Y_n - Y|)\,.
$$
\end{theorem}

For $m\in \{1,\ldots, M\}$, $k\in \{1,\ldots, K\}$, and $n\in \bbN$, define
$$
Y^m_n=\langle f_m(y)\rangle_{T}^{(1-2H),n}\quad \textnormal{and} \quad 
X^{mk}_n= \sum_{\pi_n}  \sigma_k[f_m]^2(y_{t_i}) (t_{i+1} - t_i) \,.
$$
Let $Y_n=(Y^m_n)_{1 \leq m \leq M}$,  $X_n=(X^{m,k}_n)_{1 \leq k \leq K, 1 \leq m \leq M}$,  
\begin{equation}\label{def:theta_est_hurst_known}
\beta_n=\operatorname{arg\,min}_{\beta} |Y_n - X_n \beta|\,, \quad \textnormal{and} \quad \theta_n = \sqrt{\beta_n}\,.
\end{equation}
Applying Theorem \ref{thm:least_square_stability}, we get the following: 

\begin{theorem} \label{thm:parameter rates}
Let $B$ is a fractional Brownian motion with $H \in (\frac13, 1)$ and $y$ is the pathwise solution of \eqref{eq:RDE_fBm_theta_linear}. Assume that $\bbF=\{f_1,\ldots, f_m\}\subset C_2(\bbR^d)$ is such that the columns of $X$ are linearly independent and $Y\ne \boldsymbol{0}$. Then, for all $\epsilon >0$, $\mathbb{P}$-a.s.,
$$
|\beta - \beta_n| =\mathcal{O}_{C_1} \left( \delta_n^{H-\epsilon}  \right) \,,
$$
where $C_1=C_1(T, \kappa(X), \|B\|_{H-\epsilon} ,\|f(y)'\|_{H-\epsilon},[f(y)^{\sharp}]_{2(H-\epsilon)},  [\sigma_k[f]^2(y)]_{H-\epsilon}, f\in \bbF)$. Moreover, for all $k$ such that  $\beta_k>0$ and for all $\epsilon >0$,  $\bbP$-a.s., 
\begin{equation}\label{eq:fBm convergence rates_theta}
|\theta_k - \theta_{n,k}|  =  \mathcal{O}_{C_1\vee \beta_k^{-1}} \left( \delta_n^{H-\epsilon}  \right) \,.
\end{equation}
\end{theorem}
\begin{proof}
Since $Y= X  \beta^*$, then $\beta \in \operatorname{arg\,min}_{\beta} \|Y - X \beta\|$ with $\|Y - X \beta\|=0$. The assumption $Y\ne \boldsymbol{0}$ enforces the condition \eqref{asm:general_stability_ne} of Theorem \ref{thm:least_square_stability}. Combined with the assumption of linear independence of the columns of $X$, we can apply Theorem \ref{thm:least_square_stability}. Since $t \mapsto \sigma_k[f_m]^2(y_t)$ is $H-\epsilon$ H\"older continuous and $t \mapsto t$ is $\eta$-H\"older continuous for every $\eta < 1$, the 
 sewing lemma (\cite[Lemma 4.2]{FrizHairer}) yields
\begin{gather*}
\int_0^T \sigma_k[f_m]^2(y_r) \rmd r  - \sum_{[t_i,t_{i+1}]\in \pi_n} \sigma_k[f_m]^2(y_{t_i}) (t_{i+1} - t_i) \\
=\sum_{[t_i,t_{i+1}]\in \pi_n} \left( \int_{t_{i}}^{t_{i+1}} \sigma_k[f_m]^2(y_r) \rmd r - \sigma_k[f_m]^2(y_{t_i}) (t_{i+1} - t_i) \right) \\
\leq \sum_{[t_i,t_{i+1}]\in \pi_n} [\sigma_k[f_m]^2(y)]_{H-\epsilon} [\text{id}]_{\eta} (t_{i+1} - t_i)^{H-\epsilon + \eta}  =  [\sigma_k[f_m]^2(y)]_{H-\epsilon} [\text{id}]_{\eta} 2^{-n(H-\epsilon + \eta - 1)}\,,
\end{gather*}
so long as $H-\epsilon + \eta > 1$. Choosing $\eta$ sufficiently close to 1, we find that 
$$
\|X - X_n\|_2 = \mathcal{O}_{C_1}(2^{-n(H-2\epsilon)}).
$$
Moreover, by Theorem \ref{thm:rate_for_fractional_diffusion}, $|Y-Y_n|= \mathcal{O}_{C_1}\left( \delta_n^{H-\epsilon} \right) $. By virtue of Theorem \ref{thm:least_square_stability}, we have 
\begin{align*}
|\beta - \beta_n | & = \mathcal{O}_{C_1}(2^{-n(H-2\epsilon)}) \vee \mathcal{O}_{C_1}\left( \delta_n^{H-\epsilon} \right) \\
& =\kappa(X)^2 \mathcal{O}\left( \delta_n^{H-\epsilon} \right) \,.
\end{align*}
Then, \eqref{eq:fBm convergence rates_theta} follows from the mean-value theorem applied to the square root function.
\end{proof}

We now introduce the estimation of $\theta$ when the estimated Hurst index, $H_n$, is used instead of the true $H$. Define
$
\bar{Y}^m_n=\langle f_m(y)\rangle_{T}^{(1-2H_{n}),n}
$
and let $\bar{Y}_n=(\bar{Y}^m_n)_{1 \leq m \leq M}$, 
\begin{equation}\label{def:theta_est_hurst_unknown}
\bar{\beta}_n=\operatorname{arg\,min}_{\beta} |\bar{Y}_n  - X_n \beta|\,, \quad \textnormal{and} \quad \bar{\theta}_n = \sqrt{\bar{\beta}_n}\,.
\end{equation}
Following the proof of Theorem \ref{thm:parameter rates}, using Corollary \ref{cor:double_approx_rates}, we obtain the following theorem concerning the convergence rates of $\theta$ if the Hurst index is not known.
\begin{theorem} 
Retain the assumptions and notation of Theorem \ref{thm:parameter rates}. Then, for all $\epsilon >0$,  $\bbP$-a.s., 
$$
|\beta - \bar{\beta}_n|  = \mathcal{O}_{C_1}\left( n \delta_n^{H-\epsilon}  \right) \,,
$$
where $\epsilon > 0$ can be chosen arbitrarily small. Moreover, for all $k$ such that  $\beta_k>0$ and for all $\epsilon >0$, $\bbP$-a.s., 
$$
|\theta_k - \bar{\theta}_{n,k}|  = \mathcal{O}_{C_1\vee \beta_k^{-1}} \left(n \delta_n^{H-\epsilon}  \right)\,.
$$
\end{theorem}

\begin{remark}
The proposed method may not work if the number of noise sources is larger than the dimension of the dynamics, i.e., if $K>d$. Indeed, consider $d=1$, $K=2$, and the simple dynamics
$$
\rmd y_t = \theta_1 \rmd B_t^{1} + \theta_2 \rmd B_t^2, \qquad y_0 = 0.
$$
Since $y_t \sim \mathcal{N}(0, (\theta_1^2 + \theta_2^2) t^{2H})$, we should be able to estimate $\theta_1^2 + \theta_2^2$, but not the individual terms in the sum. To make this rigorous, observe that
$$
\rmd f(y_t) = \theta_1 f'(y_t) \rmd B_t^{1} + \theta_2 f'(y_t) \rmd B_t^2,
$$
and thus for all $f \in C^2(\bbR^d)$
$$
\langle f(y) \rangle_t^{(1-2H)} = \theta_1^2  \int_0^t f'(y_r) \rmd r + \theta_2^2  \int_0^t f'(y_r) \rmd r\,.
$$
We see then that for any $\{f_1, \dots, f_M\} \subset C^2(\bbR^d)$,  the matrix defined in \eqref{eq:def least square matrices} above satisfies
$$
X^{m1} =X^{m2} = \int_0^T f_m'(y_r) \rmd r\,,
$$
and hence $\textnormal{rank}(X)<2$ which implies that $X^TX$ is not invertible and $\kappa(X) = \infty$. 

See, however, Section \ref{sec:1dnonlinear} for an example in which we can find a set of functions for which the condition number is small.  
\end{remark}

\section{Experiments}

In each of our experiments, we consider a rough differential equation (RDE)
\begin{equation}\label{eq:rde_example_fBm}
\rmd y_t = u_t(y_t) \rmd t + \sum_{k=1}^K \theta_{k} \sigma_k(y_t) \rmd \bB_t^k\,, \;\; t\in (0,1], \quad y|_{t=0}=y_0 \in \mathbb{R}^d\,,
\end{equation}
where $\beta, \sigma_1, \ldots, \sigma_K: \bbR^d\rightarrow \bbR^d$ are smooth vector fields, $\bB$ is the geometric lift of a $K$-dimensional fractional Brownian paths obtained by a limit of a lifted piecewise linear approximation with Hurst parameter  $H\in \{0.35, 0.4, 0.45, 0.5, 0.6, 0.7\}$.

To obtain a realization of a solution of \eqref{eq:rde_example_fBm} for a fixed initial condition, we simulate a realization of an $\bbR^K$-valued fBm on a uniform dyadic partition $\pi_f$ of $[0,1]$ using the Davies-Harte method \cite{davies1987tests} and then compute either an explicit solution or an approximate solution if an explicit solution does not exist. In particular, we compute an approximate solution using a third-order Runge-Kutta method on $\pi_f$ satisfying the conditions of Table 1 of \cite{redmann2022runge}. By Theorem 4.2 in \cite{redmann2022runge}, such approximate solutions converge to the solution of \eqref{eq:rde_example_fBm} in the supremum norm with rate $\mathcal{O}(|\pi_f|^{2H-\frac{1}{2}-\varepsilon})$ for any $\varepsilon>0$. 

To evaluate the convergence rate as a function of the mesh size, we sub-sample the solutions on uniform dyadic partitions $\pi_n$ of the interval $[0,1]$ with mesh sizes $|\pi_n|=2^{-n}$, $n\in \{2, 3,\ldots, f\}$. For each realization and sub-sampled solution on $\pi_n$, we estimate 
\begin{itemize}
\item the Hurst parameter $H_n$ using \eqref{def:hurst_estimator};
\item the parameter $\theta_n$ using \eqref{def:theta_est_hurst_known};
\item the parameter $\bar{\theta}_n$ using \eqref{def:theta_est_hurst_unknown}
\end{itemize}
using a family of functions $\bbF$ such that $X_n$ has low condition number. For each example, we give a box plot of  $\log_2\left(\frac{H-H_n}{H\delta_n^H}\right)$, $\log_2\left(\frac{|\theta-\theta_n|}{\delta_n^H}\right)$, and $\ln_2\left(\frac{|\theta-\bar{\theta}_n|}{n\delta_n^H}\right)$ for each $n=\log_2(|\pi_n|^{-1})$ and $H$. We normalize the error rate for the Hurst estimation, viz $\log_2\left(\frac{H-H_n}{H\delta_n^H}\right)$, since we are not keeping track of the the constants appearing in the convergence rates. 

In the box-plots the upper and lower extremes of the box represents the 75th and 25th percentiles respectively, and the line inside the box represents the median. The whiskers encapsulates the extreme values of the data points, and in particular, we note that there are no outliers in the plots.

The code for the experiments was written in Google JAX, making extensive use of the Equinox \cite{kidger2021equinox} and Diffrax \cite{kidger2021on} libraries developed by Patrick Kidger and others.

In each example, our plots suggest that our convergence rates are not sharp across all Hurst-indicies. Thus, we leave finding the optimal convergence rate as an open problem. 

\subsection{1d non-linear}\label{sec:1dnonlinear}

In this example, we consider the RDE given by
$$
\rmd y_t = 0.5\sin(y_t)\rmd \bB^1_t + 0.8 \cos(y_t)\rmd \bB^2_t\,, \;\; t\in (0,1], \quad y|_{t=0}=1
$$
and use 1000 realizations.
Comparing with \eqref{eq:rde_example_fBm}, we take $\theta = [0.5 \; \; 0.8]^{\top}$, $\sigma_1(y)=\sin(y)$, and $\sigma_2(y)=\cos(y)$. We approximate the solution using Heun's third-order method with $|\pi_f|=2^{-17}$. Based on experiments, we found that the choice $\bbF=\{f_1(y)=\sin(y), f_2(y)=\cos(y)\}$ leads to an $X_n$ with a low condition number. 

\begin{figure}[H]
\centering
\caption{Plot of ``1d non-linear'' rate of convergence of Hurst estimation}  
\includegraphics[width=10cm]{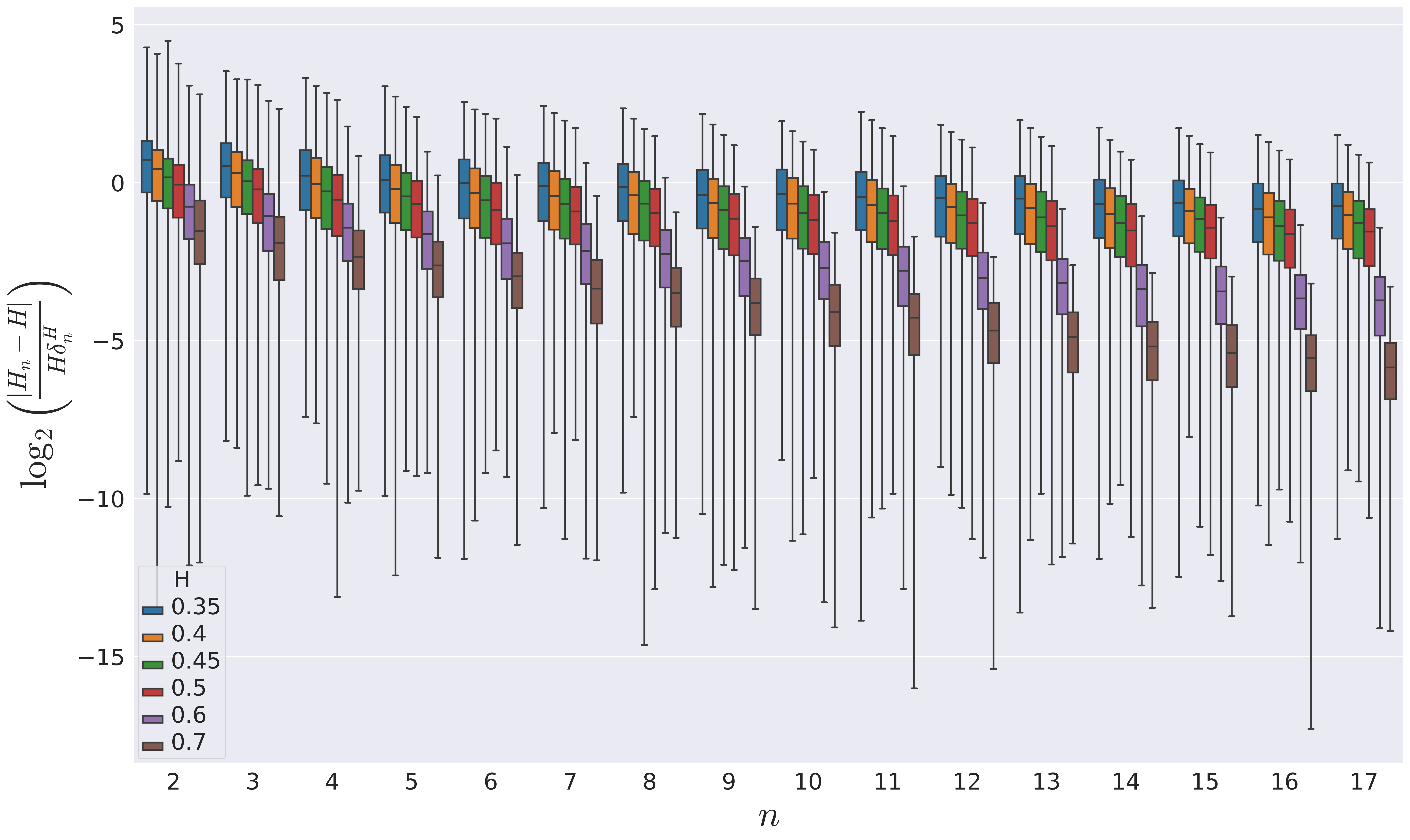}
\end{figure}

\begin{figure}[H]
\centering
\caption{Plot of ``1d non-linear'' rate of convergence of theta estimation if $H$ is known}  
\includegraphics[width=10cm]{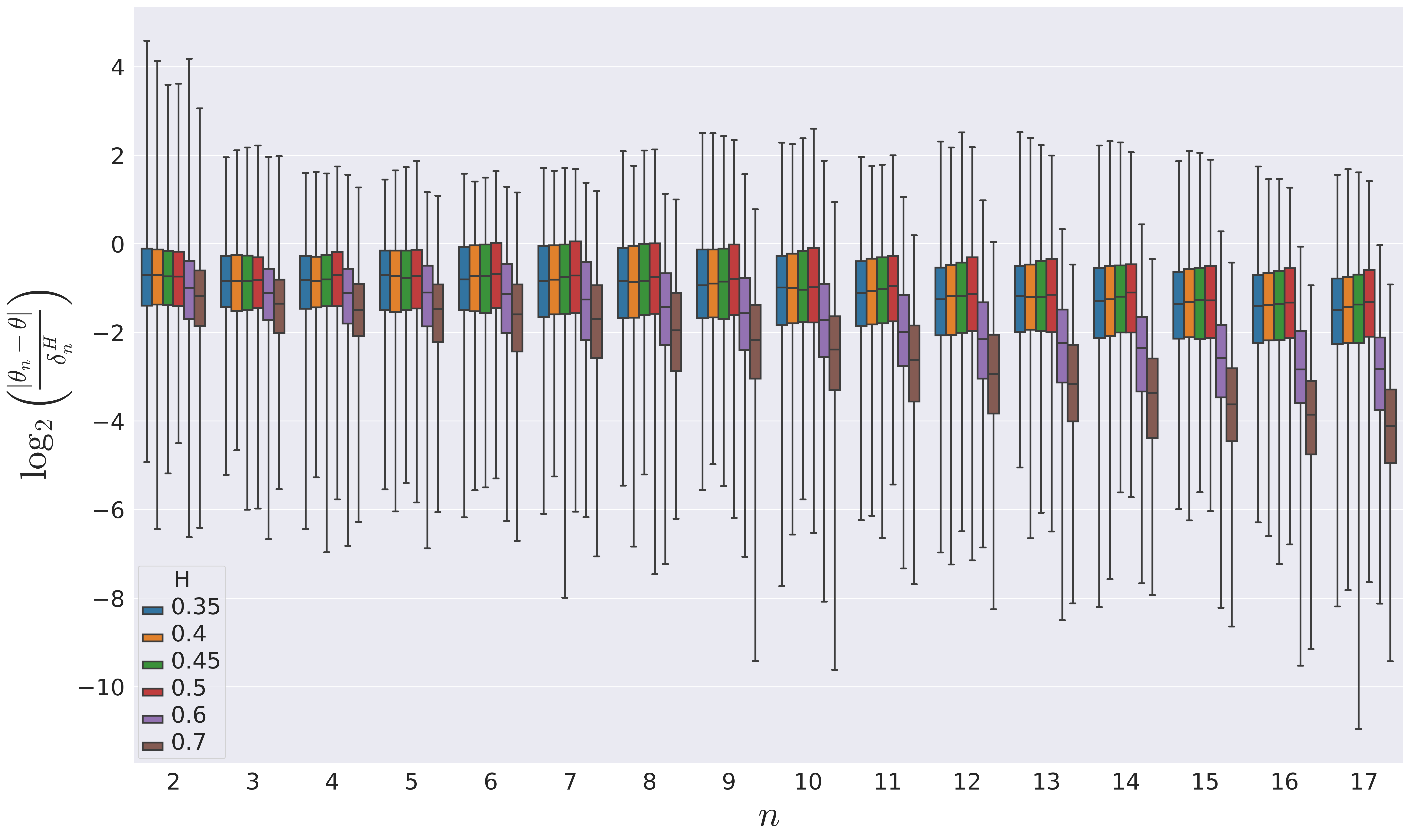}
\end{figure}

\begin{figure}[H]
\centering
\caption{Plot of ``1d non-linear'' rate of convergence of theta estimation if $H$ is unknown} 
\includegraphics[width=10cm]{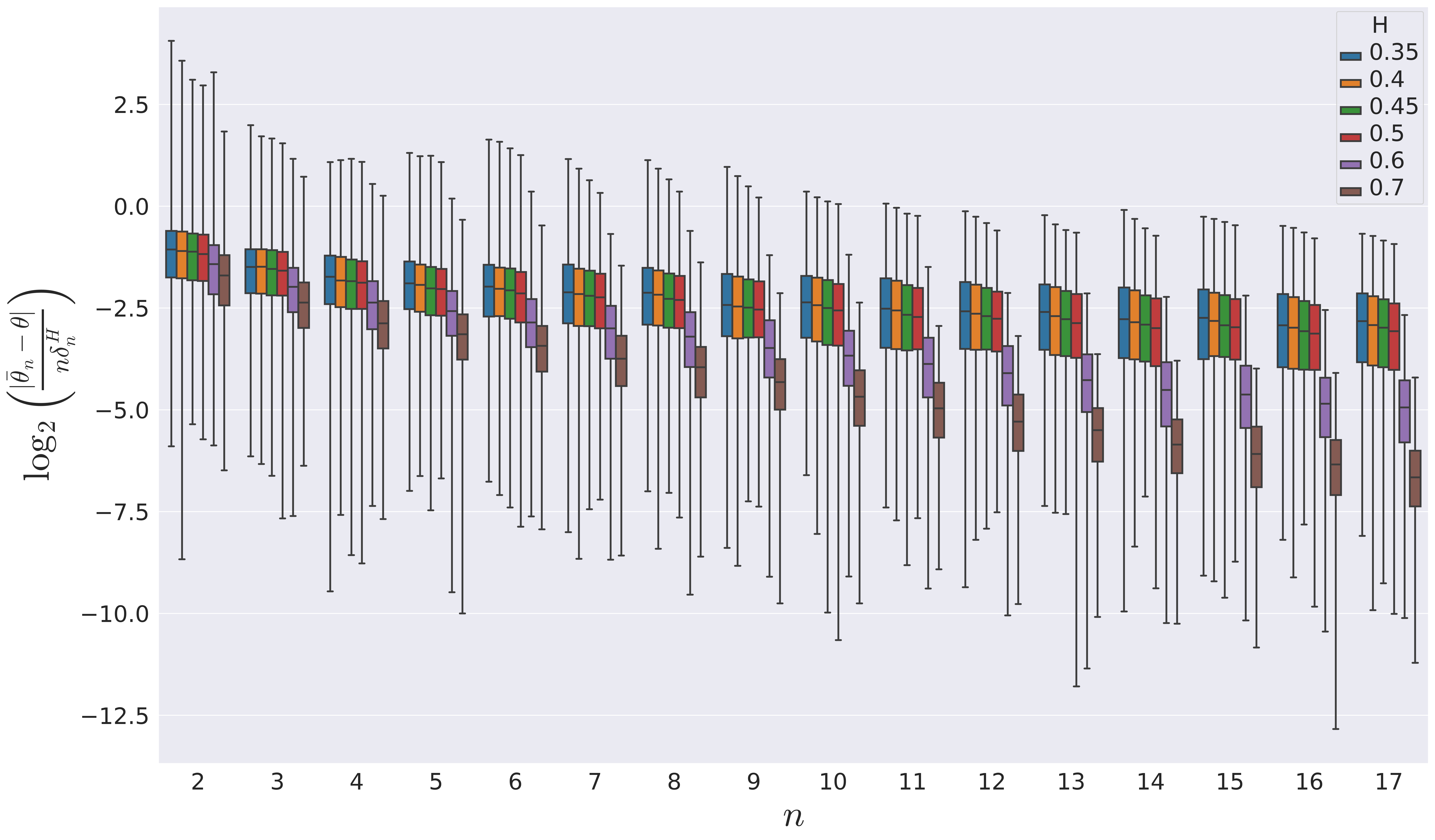}
\end{figure}

\subsection{2d linear}

Borrowing from the experiments in \cite{redmann2022runge}, in this example, we consider the RDE given by
$$
\rmd y_t = 0.1A_1y_t \rmd \mathbf{B}^1_t +0.1 A_2y_t \rmd \mathbf{B}^2_t\,, \;\; t\in (0,1], \quad y|_{t=0}=[0.5 \;\; 1]^{\top}\,,
$$
where $A_1 = \begin{bmatrix}
1 & 0.5 \\
0.5 & 1
\end{bmatrix}$ and $A_2 = \begin{bmatrix}
2 & -1 \\
-1 & 2
\end{bmatrix}$ and use 1000 realizations. Comparing with \eqref{eq:rde_example_fBm}, we take $\theta = [0.1 \; \; 0.1]^{\top}$, $\sigma_1(y)=A_1 y$, and $\sigma_2(y)=A_2y$. Since $0.1 A_1$ and $0.1 A_2$ commute, the analytic solution is given by $y_t= \exp\left(0.1 A_1 B^1_t + 0.1A_2 B^2_t\right)y_0$. Based on experiments, we found that the choice
$\bbF=\{f_1(y^1, y^2)=y^1, f_2(y)=y^2\}$ leads to an $X_n$ with a low condition number. 

\begin{figure}[H]
\centering
\caption{Plot of ``2d linear'' rate of convergence of Hurst estimation}  
\includegraphics[width=10cm]{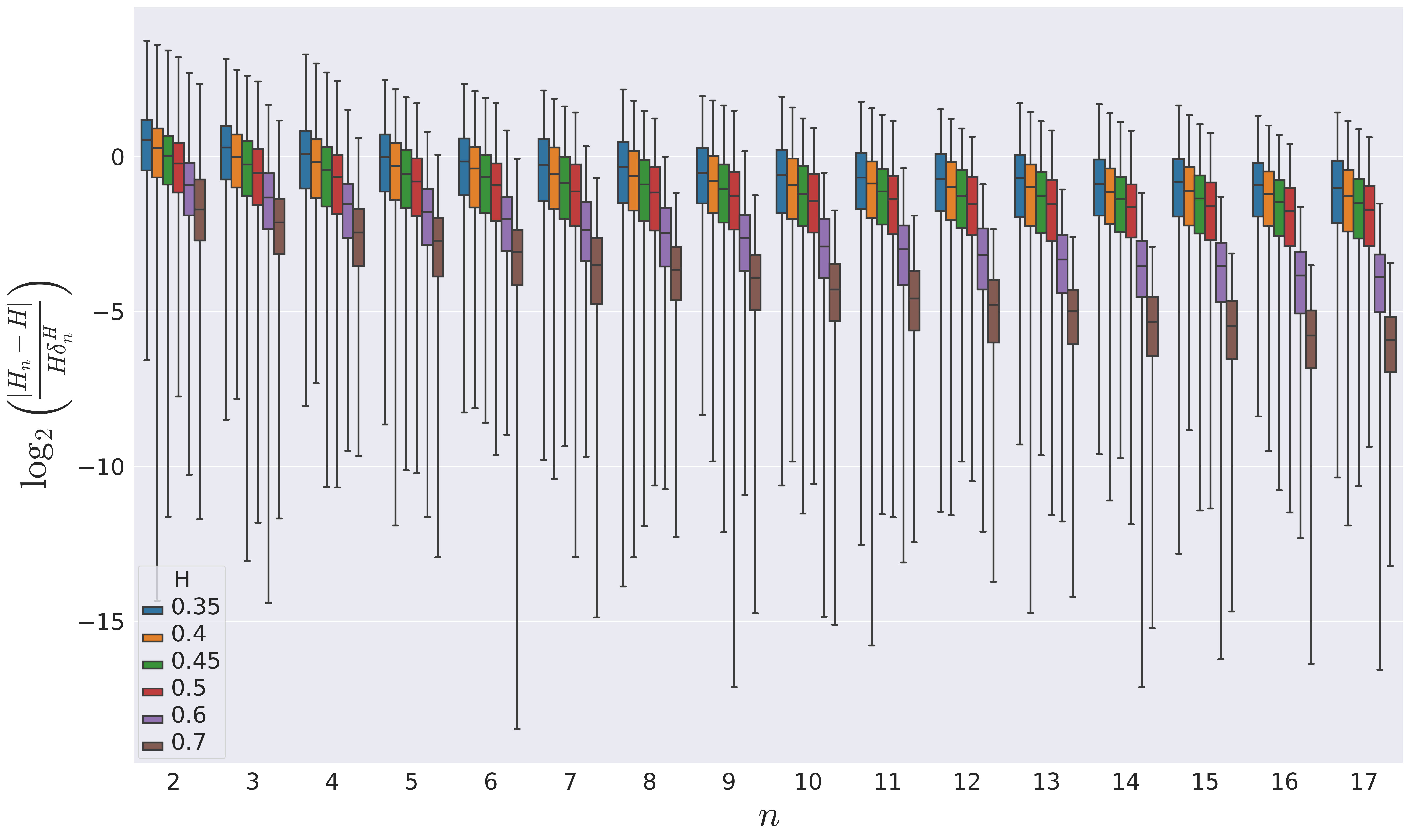}
\end{figure}

\begin{figure}[H]
\centering
\caption{Plot of ``2d linear'' rate of convergence of theta estimation if $H$ is known}  
\includegraphics[width=10cm]{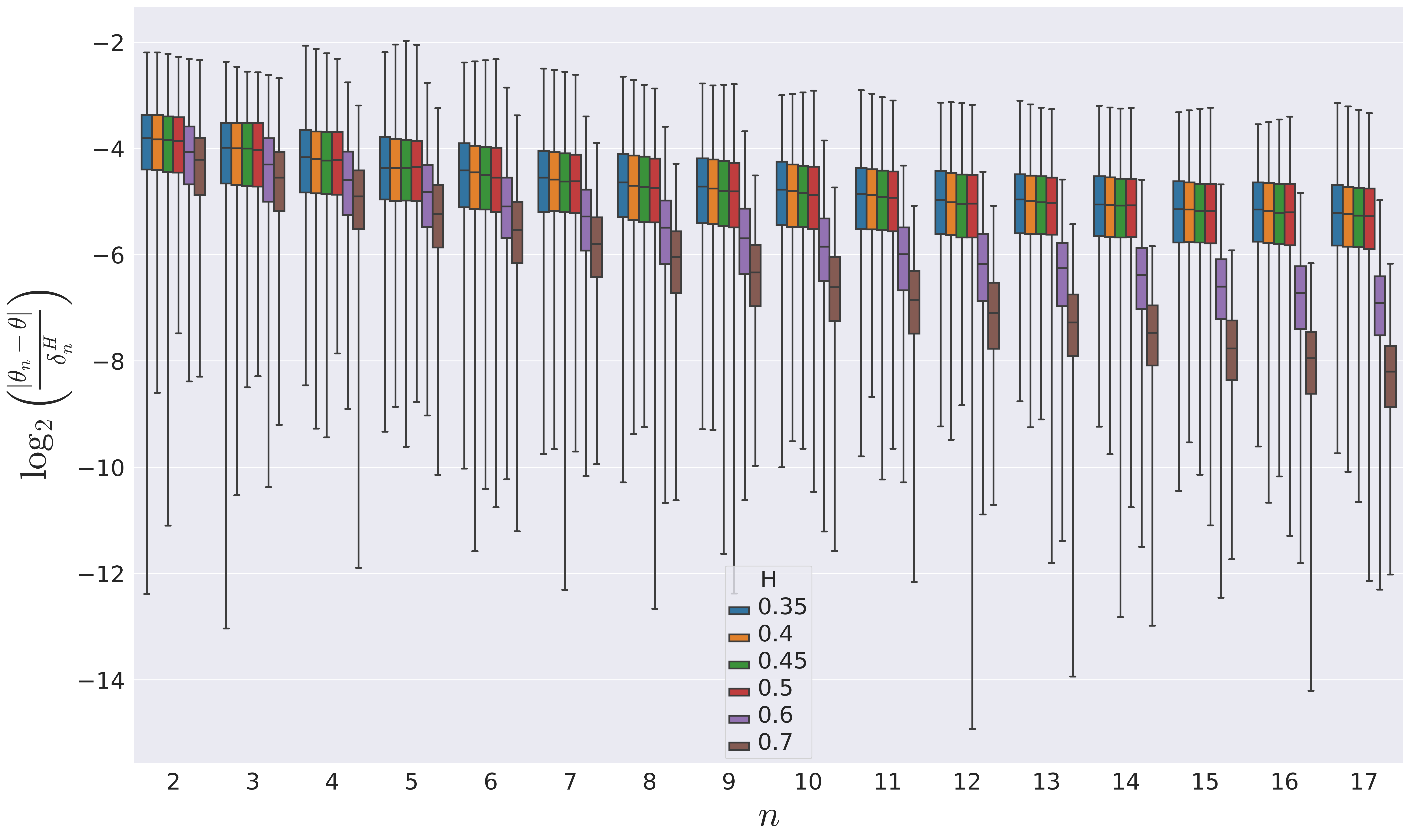}
\end{figure}

\begin{figure}[H]
\centering
\caption{Plot of ``2d linear'' rate of convergence of theta estimation if $H$ is unknown} 
\includegraphics[width=10cm]{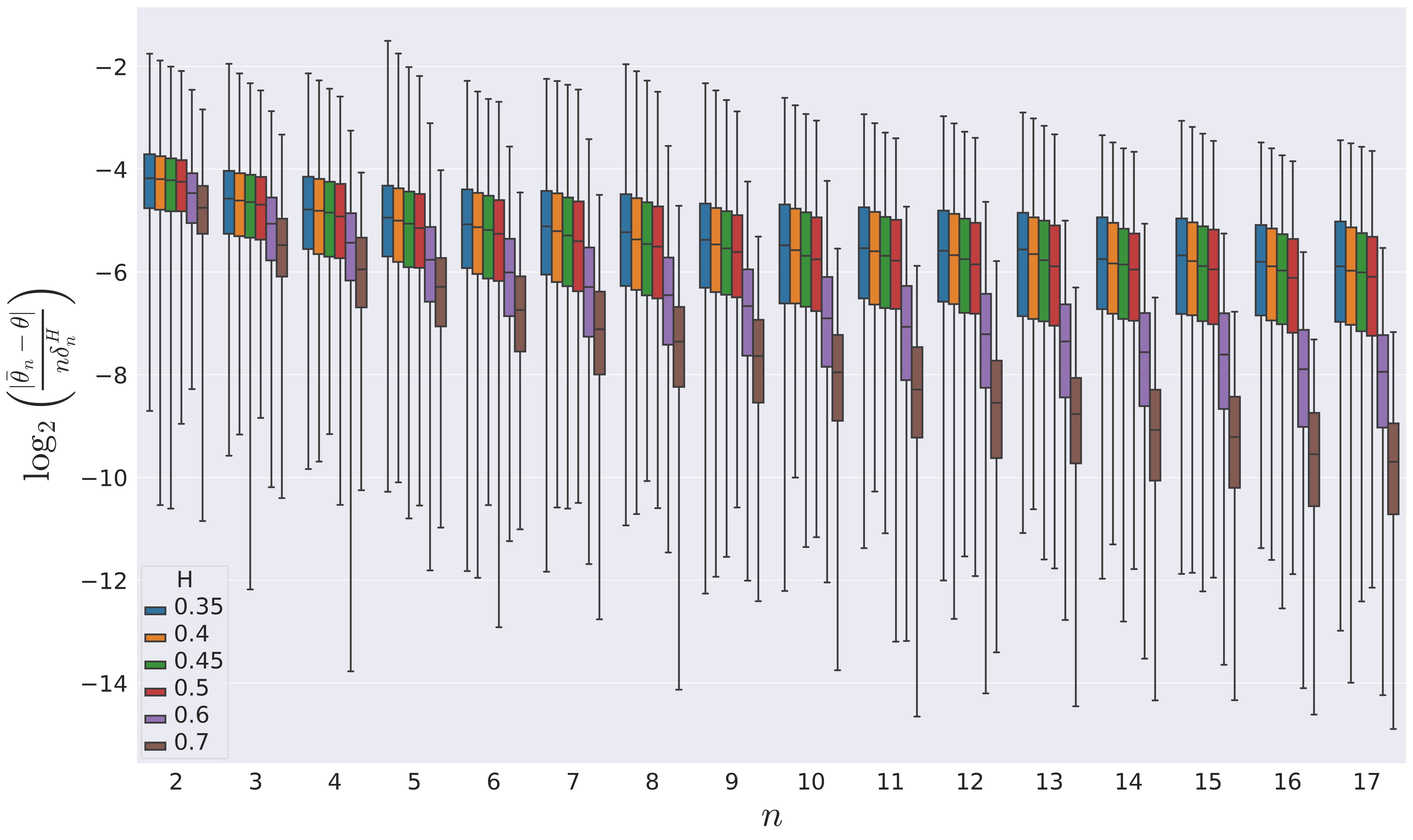}
\end{figure}

\subsection{2d non-linear}
In this example, we consider the RDE given by
$$
\begin{cases}
\rmd x_t = 0.5 \sin(x_t)\cos(y_t)\rmd \bB^1_t\,, \\
\rmd y_t = 0.8 \cos(x_t)\sin(y_t)\rmd \bB^2_t\,,
\end{cases} t\in (0,1], \quad y|_{t=0}=[0.5 \;\; 0.5]^{\top}\,,
$$
and use 1000 realizations. Comparing with \eqref{eq:rde_example_fBm}, we take $\theta = [0.5 \; \; 0.8]^{\top}$, $\sigma_1(x,y)=[\sin(x)\cos(y) \; \; 0]^{\top}$, and $\sigma_2(x,y)=[0 \; \; \cos(x)\sin(y)]^{\top}$. We compute an approximate solution using Heun's third-order method.  Based on experiments,  we found that the choice $\bbF=\{f_1(x,y)=x, f_2(x, y)=y\}$ leads to an $\bbX_n$ with a low condition number. 

\begin{figure}[H]
\centering
\caption{Plot of ``2d non-linear'' rate of convergence of Hurst estimation}  
\includegraphics[width=10cm]{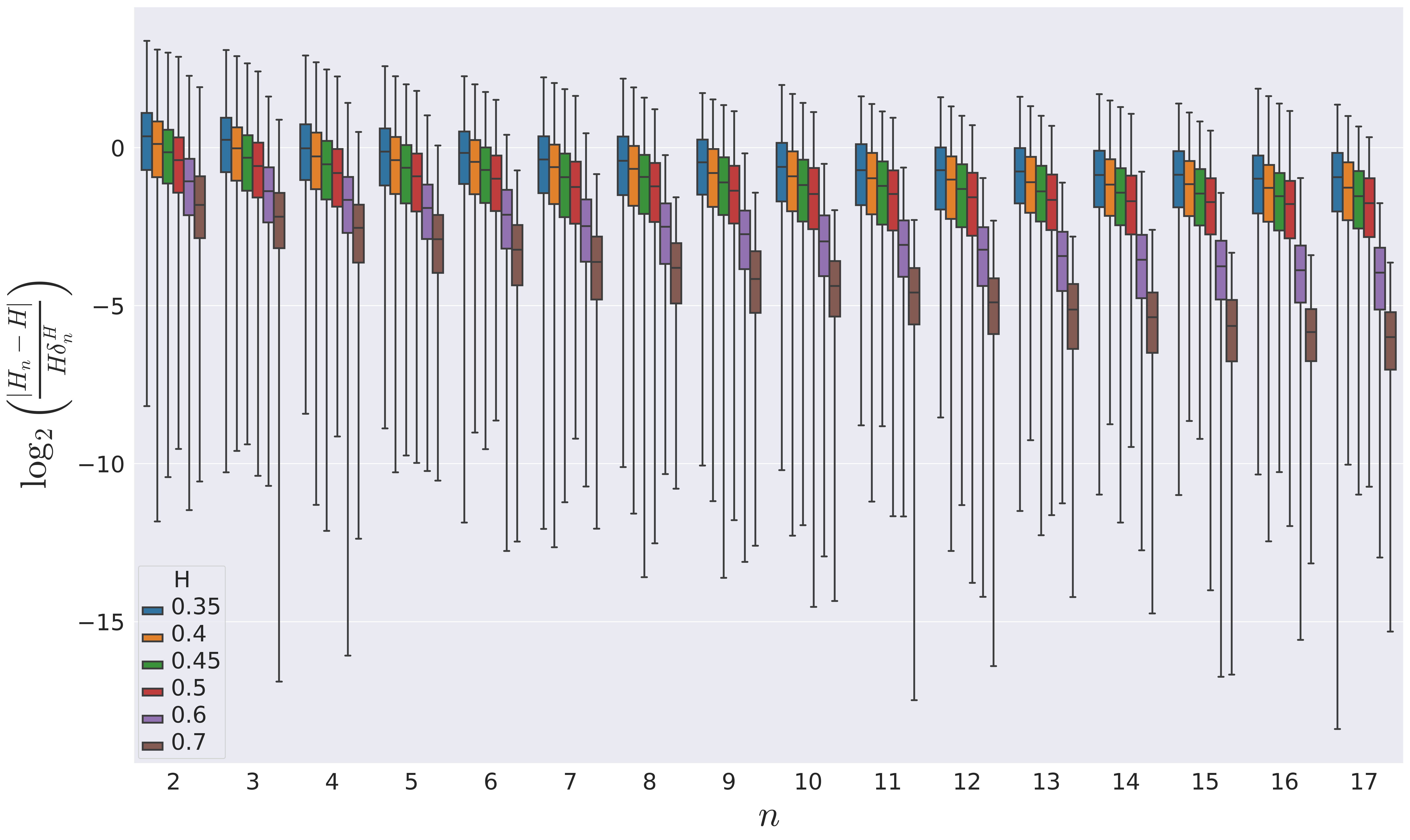}
\end{figure}

\begin{figure}[H]
\centering
\caption{Plot of ``2d non-linear'' rate of convergence of theta estimation if $H$ is known}  
\includegraphics[width=10cm]{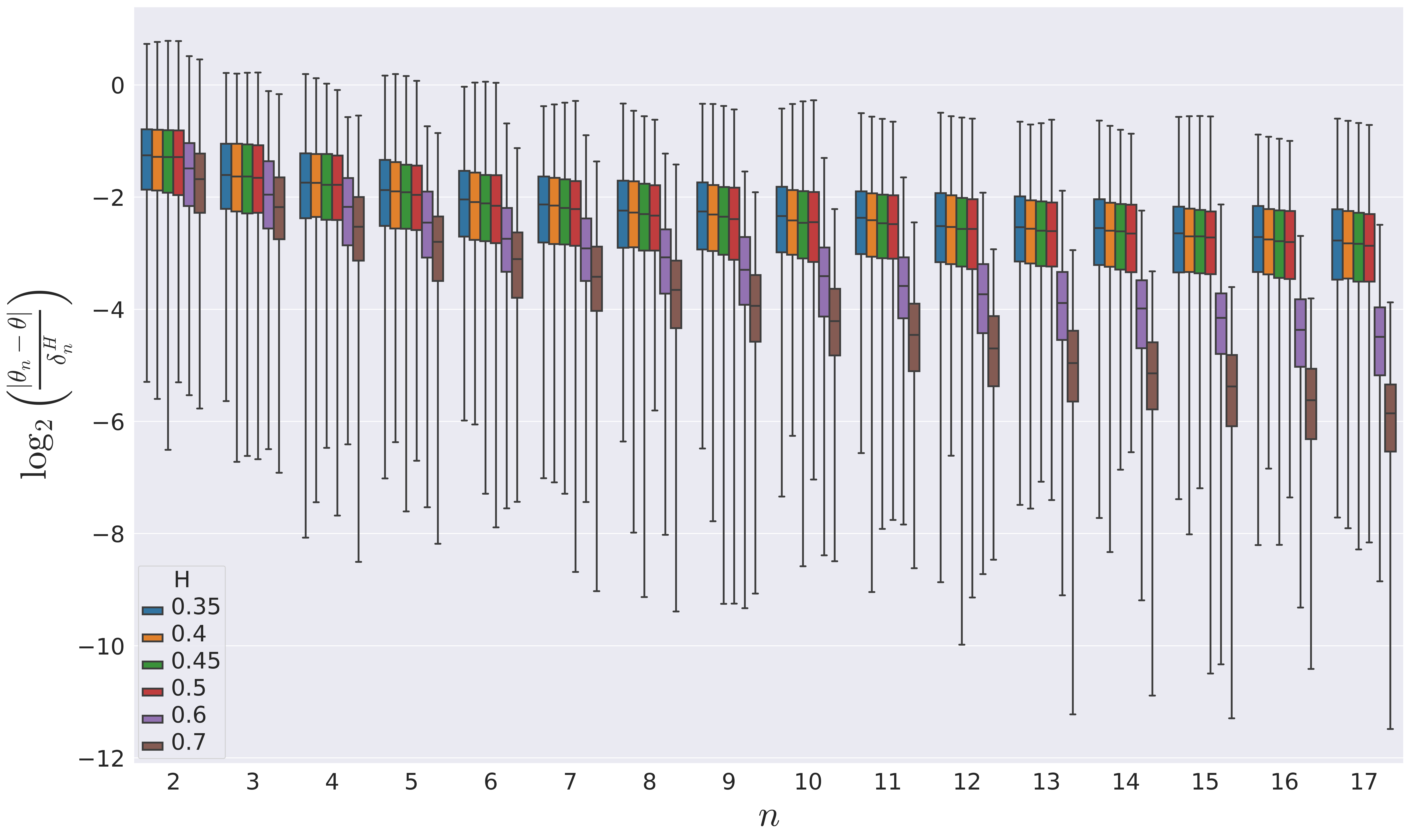}
\end{figure}

\begin{figure}[H]
\centering
\caption{Plot of ``2d non-linear'' rate of convergence of theta estimation if $H$ is unknown} 
\includegraphics[width=10cm]{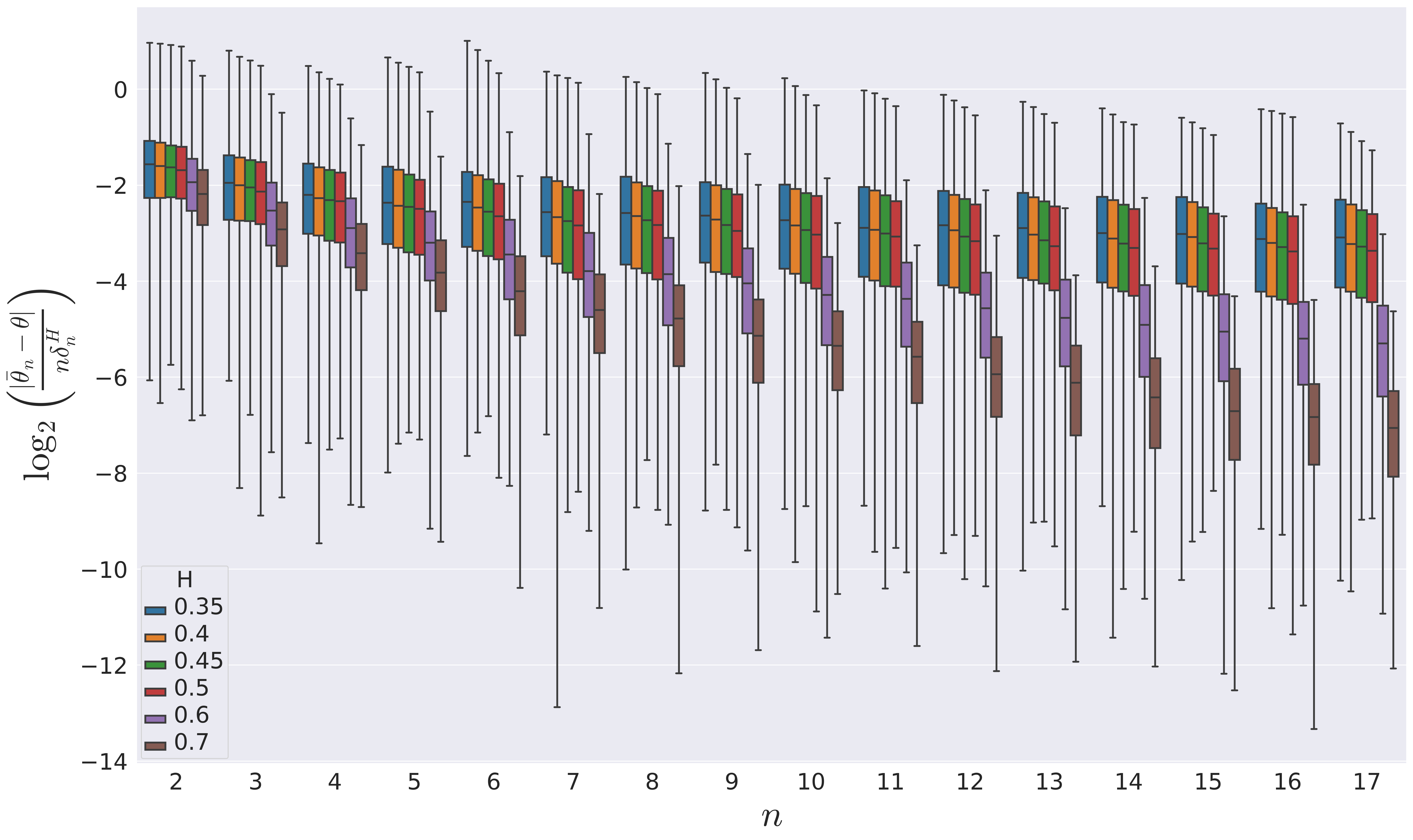}
\end{figure}

\subsection{2d Euler system with Lie transport noise (fluid RPDE)}\label{sec:Euler_system}
In this example, we consider a rough partial differential equation (RPDE) governing a two-dimensional turbulent fluid flow on the flat torus $\mathbb{T}^2=[0,2\pi]^2/\sim$. Specifically, we consider a forced and damped Euler system with Lie transport fBm noise in its vorticity formulation:
\begin{equation}\label{eq:euler2d_vform}
\begin{cases}
\rmd \omega_t + (u_t\cdot \nabla \omega_t + 0.1\sin(8x) - 0.01 \omega_t) \rmd t + \sum_{k=1}^8\theta_k\xi_k \cdot \nabla \omega_t \rmd \bB_t^k=0\,, \quad t\in (0, 1]\,,\\
u_t = \Delta^{-1}\nabla^{\perp}\omega_t\,, \quad t\in (0,1]\,,\\
\omega|_{t=0}=\phi\,,
\end{cases}
\end{equation}
where
\begin{itemize}
\item $u_t(x,y)$ represents the spatial velocity of the fluid at $(x,y)\in \bbT^2$;
\item $\omega_t(x,y)=\nabla^{\perp}u_t(x,y)$ represents the vorticity of the fluid at $(x,y)\in \bbT^2$;
\item $\Delta = \partial_x^2 + \partial_y^2$ denotes the Laplacian, $\nabla^{\perp}= -\partial_{y} \boldsymbol{i} + \partial_x \boldsymbol{j}$ denotes the curl operator, and thus $\Delta^{-1}\nabla^{\perp}$ is the Biot-Savart operator on $\bbT^2$;
\item $\{\xi_k\}_{k=1}^8=\{\nabla^{\perp}\alpha_k\}_{k=1}^8$ are a class of smooth divergence-free vector fields, where the potential functions are given by
$$
\alpha_1(x,y) = \sin(y), 
\quad \alpha_2(x,y)=\cos(y), 
\quad \alpha_3(x,y)=\sin(x), 
\quad \alpha_4(x,y)=\cos(x),
$$
$$
\alpha_5(x,y)=\sin(y), 
\quad \alpha_6(x,y)=\cos(y), 
\quad \alpha_7(x,y)=\sin(x+y), \quad \alpha_8(x,y)=\cos(x+y);
$$
\item  $\theta_k=0.01$ for all $k\in \{1,\ldots, 8\}$;
\item $\phi$ is a mean-free smooth initial condition.
\end{itemize}
The introduction of Equation \eqref{eq:euler2d_vform}, without the forcing and damping, but with a general rough path, appears in \cite{crisan2022variational}, where the authors also derived its variational characterization. Furthermore, \cite{crisan2022solution, hofmanova2021rough} established the well-posedness of the rough, unforced, and un-damped system, both with and without viscosity.  In \cite{cotter2019numerically}, the authors studied Equation \eqref{eq:euler2d_vform} with slip-flow boundary conditions numerically in the Brownian case  (i.e., $H=1/2$) to represent a filtered and projected solution of a deterministic forced and damped Euler system derived from a high-resolution simulation. In \cite{lang2021pathwise}, the authors applied a similar estimation methodology using quadratic variation to estimate parameters of the unforced and un-damped fluid equation in the Brownian case.

A minor extension of \cite[Theorem 3.10]{crisan2022solution} shows that the system \eqref{eq:euler2d_vform} is globally well posed in Sobolev spaces $W^{m,2}(\mathbb{T}^2)$ with arbitrary integer regularity $m\ge 2$. Indeed, the addition of the forcing and damping terms in the drift play a minor role in terms of the well-posedness analysis. We refer to \cite{crisan2022solution} for a precise definition of a solution. However, we note that for a fixed test function $\eta \in C^{\infty}(\mathbb{T}^2)$, the scalar-valued path $t \mapsto (\omega_t, \eta)_{L^2(\bbT^2)} := \int_{\mathbb{T}^2} \omega_t(x,y) \eta(x,y)\rmd x \rmd y$ is controlled by $B$ and its Gubinelli derivative is given by
$$
(\omega_t, \eta)_{L^2(\bbT^2)}' = \left[ \theta_1 (\omega_t, \xi_1\cdot \nabla \eta)_{L^2(\bbT^2)}  \dots  \theta_8 (\omega_t, \xi_8\cdot \nabla \eta)_{L^2(\bbT^2)} \right]\in \bbR^8\,.
$$
Thus, by Theorem \ref{thm:control_consistency}, for all $t\in [0,1]$,
\begin{equation}\label{eq:fluid_scaled_variation}
\langle (\omega_t, \eta)_{L^2(\bbT^2)}\rangle^{(1-2H)}_t = \sum_{k=1}^8 \theta_k^2 \int_0^t  (\omega_r, \xi_k\cdot \nabla \eta)_{L^2(\bbT^2)}^2 \rmd r\,.
\end{equation}

We discretize \eqref{eq:euler2d_vform} in space using a pseudo-spectral method with wave numbers $|n|\le N=2^5$, $n\in \bbZ^2$. The drift and noise convection terms in \eqref{eq:euler2d_vform} are discretized using the conservative form, applying dealiasing with the 2/3-rule before taking the products in physical space \cite{orszag1971elimination}. The spatially discretized equation can viewed as an RDE (e.g., \eqref{eq:rde_example_fBm}) and is discretized in time using an optimal third-order low storage third-order strong stability-preserving (SSP) Runge-Kutta scheme with step size $|\pi_f|=2^{-15}$ \cite{gottlieb2001strong}. As described in \cite{hussaini1987spectral, lacasce1996baroclinic}, we also apply an exponential 2/3rd cut-off filter to the solution at the end of each time step to prevent enstrophy from building up at high wave numbers. 

With minor modifications of the proof of \cite[Theorem 3.10]{crisan2022solution}, it is easy to see that this scheme is consistent in the sense that, passing to the limit as $|\pi_f|\downarrow 0$ and $N\rightarrow \infty$, the numerical approximation converges to the true solution in the $W^{m,2}$-norm uniformly in time $t\in [0,1]$.\footnote{As far as we know, the problem of obtaining joint space-time convergence rate is open.} 

A single initial condition for our experiments was obtained by computing a single realization of a solution with the initial condition $\omega(x,y)=\sin(x) + \sin(y)$ for $t\in [-200, 0]$ and $H=0.5$. This field is then is used as the initial condition for all $H\in \{0.35, 0.4, 0.45, 0.5, 0.6, 0.7\}$. 

Based on experiments we found that the choice 
$$
\bbF=\{
f_{n_1,n_2}(\omega)=(\omega, \cos([n_1 \; n_2]^{\top}\cdot))_{L^2(\mathbb{T}^2)}\}_{n_1,n_2=1}^7\,,
$$  
which is simply the lowest 49 components of the Fourier cosine series of $\omega$, leads to an $X_n$ with a low condition number. The $L^2$-inner-product on the right-hand-side of \eqref{eq:fluid_scaled_variation} is computed numerically using the fast Fourier transform just as the convection terms are computed in the pseudo-spectral scheme. For this example, we used 500 realizations. 

\begin{figure}[H]
\centering
\caption{Plot of ``2d Euler'' rate of convergence of Hurst estimation}  
\includegraphics[width=10cm]{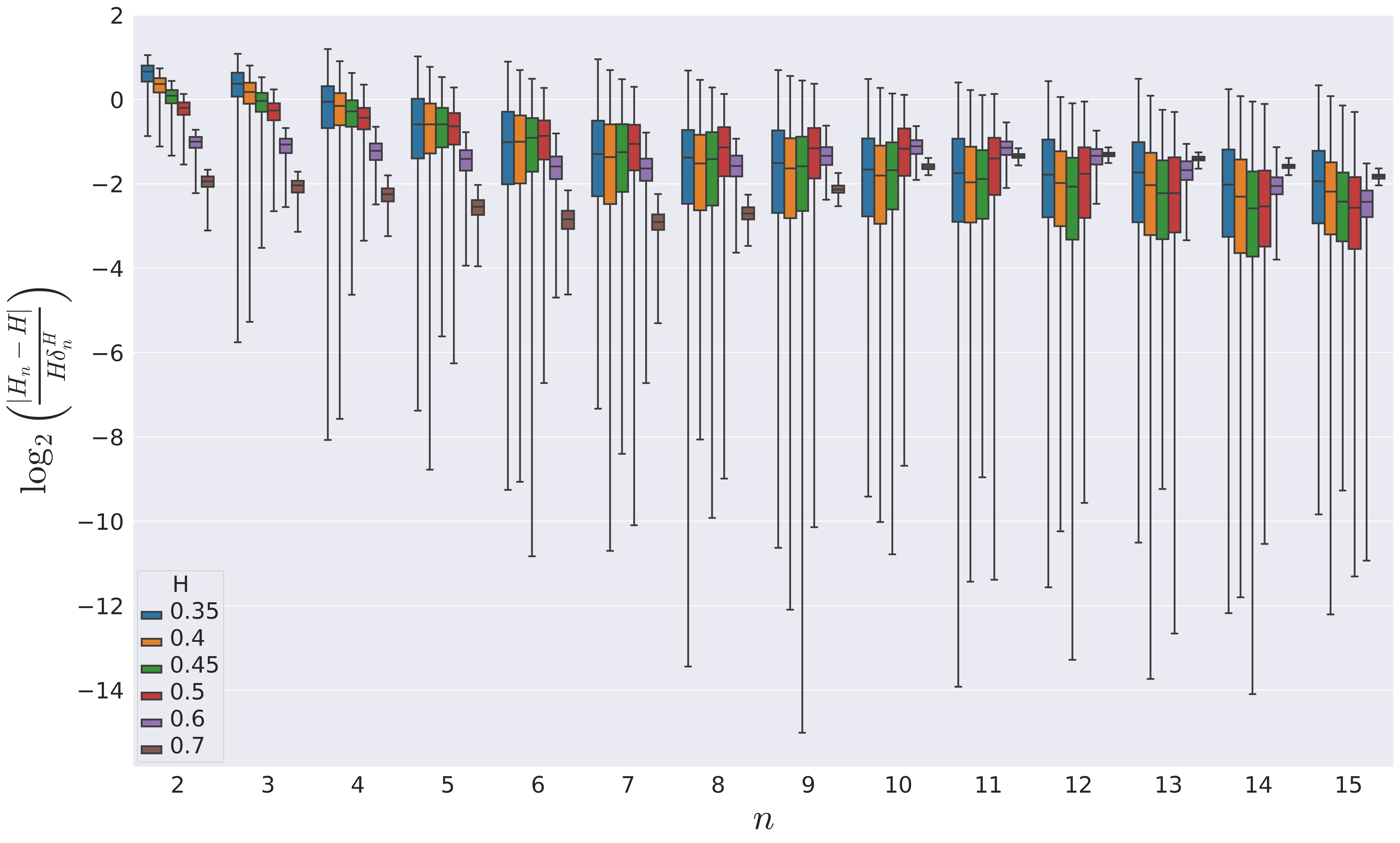}
\end{figure}
\begin{figure}[H]
\centering
\caption{Plot of ``2d Euler'' rate of convergence of theta estimation if $H$ is known}  
\includegraphics[width=10cm]{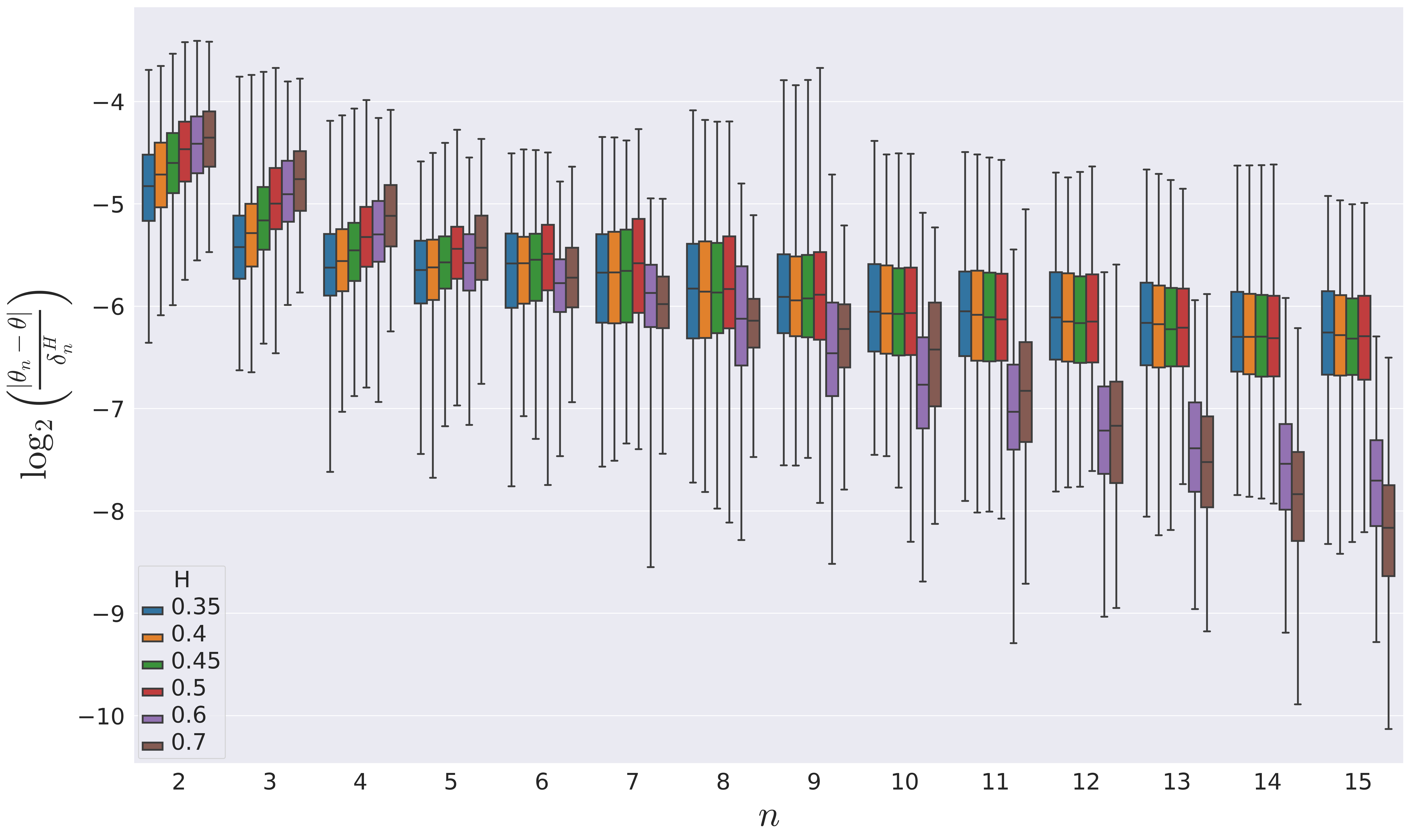}
\end{figure}
\begin{figure}[H]
\centering
\caption{Plot of ``2d Euler'' rate of convergence of theta estimation if $H$ is unknown}  
\includegraphics[width=10cm]{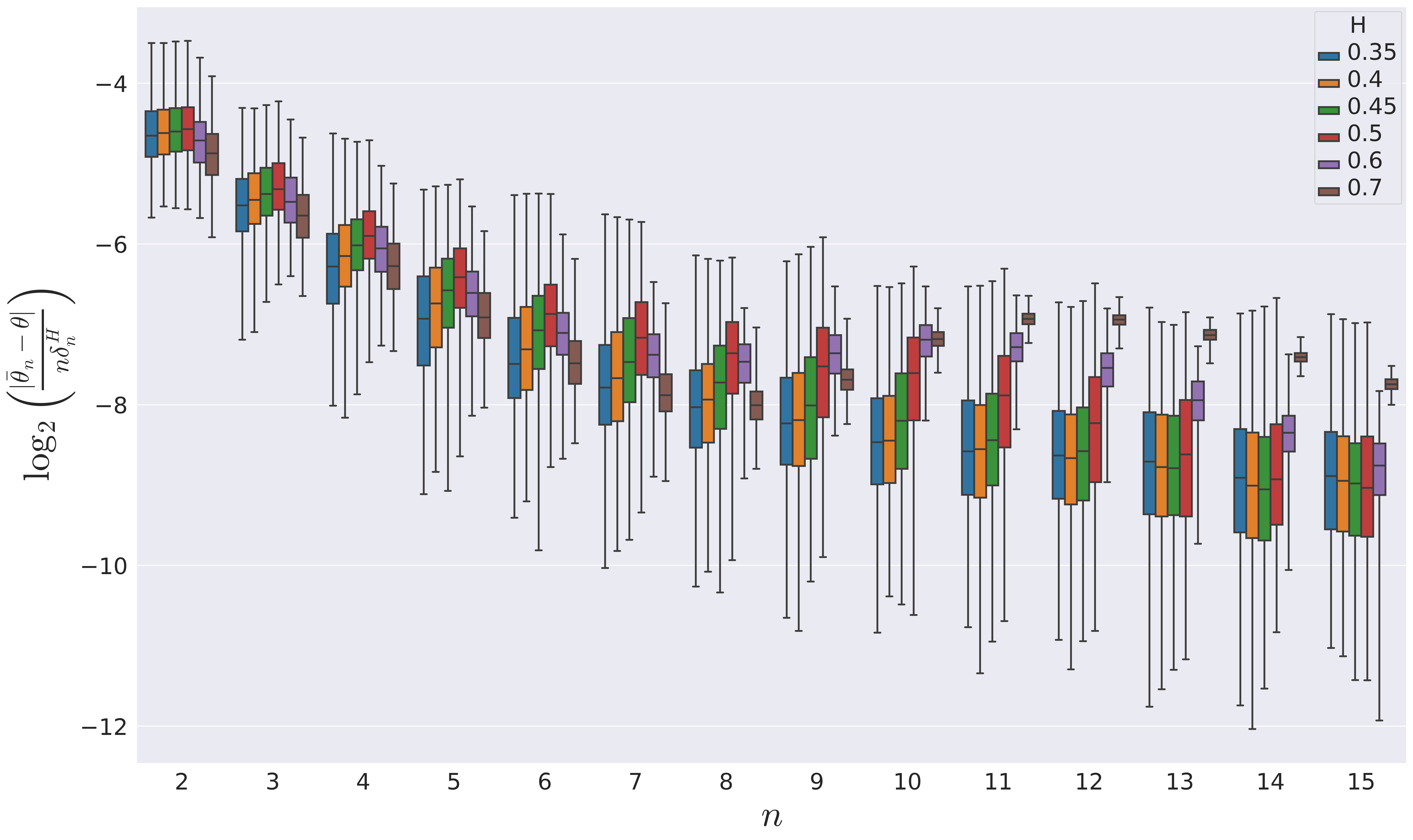}
\end{figure}

\bibliographystyle{siam}
\bibliography{bibliography}

\begin{thebibliography}{10}

\bibitem{BFGMS}
{\sc C.~Bayer, P.~K. Friz, P.~Gassiat, J.~Martin, and B.~Stemper}, {\em A
  regularity structure for rough volatility}, Mathematical Finance, 30 (2020),
  pp.~782--832.

\bibitem{Begyn}
{\sc A.~Begyn}, {\em {Quadratic Variations along Irregular Subdivisions for
  {G}aussian Processes}}, Electronic Journal of Probability, 10 (2005), pp.~691
  -- 717.

\bibitem{coeurjolly2001estimating}
{\sc J.-F. Coeurjolly}, {\em Estimating the parameters of a fractional brownian
  motion by discrete variations of its sample paths}, Statistical Inference for
  stochastic processes, 4 (2001), pp.~199--227.

\bibitem{cont2022rough}
{\sc R.~Cont and P.~Das}, {\em Rough volatility: Fact or artefact?}, Sankhya B,
  86 (2024), pp.~191--223.

\bibitem{cont2019pathwise}
{\sc R.~Cont and N.~Perkowski}, {\em Pathwise integration and change of
  variable formulas for continuous paths with arbitrary regularity},
  Transactions of the American Mathematical Society, Series B, 6 (2019),
  pp.~161--186.

\bibitem{powervariation}
{\sc J.~M. Corcuera, D.~Nualart, and J.~H. Woerner}, {\em {Power variation of
  some integral fractional processes}}, Bernoulli, 12 (2006), pp.~713 -- 735.

\bibitem{cotter2019numerically}
{\sc C.~Cotter, D.~Crisan, D.~D. Holm, W.~Pan, and I.~Shevchenko}, {\em
  Numerically modeling stochastic {L}ie transport in fluid dynamics},
  Multiscale Modeling \& Simulation, 17 (2019), pp.~192--232.

\bibitem{crisan2022solution}
{\sc D.~Crisan, D.~D. Holm, J.-M. Leahy, and T.~Nilssen}, {\em Solution
  properties of the incompressible {E}uler system with rough path advection},
  Journal of Functional Analysis, 283 (2022), p.~109632.

\bibitem{crisan2022variational}
\leavevmode\vrule height 2pt depth -1.6pt width 23pt, {\em Variational
  principles for fluid dynamics on rough paths}, Advances in Mathematics, 404
  (2022), p.~108409.

\bibitem{davies1987tests}
{\sc R.~B. Davies and D.~Harte}, {\em Tests for {H}urst effect}, Biometrika, 74
  (1987), pp.~95--101.

\bibitem{Follmer}
{\sc H.~Follmer}, {\em Calcul d'{I}to sans probabilit{\'e}s; s{\'e}minaire de
  probabilit{\'e}s xv}, Lecture Notes in Mach, 850 (1981).

\bibitem{FrizHairer}
{\sc P.~K. Friz and M.~Hairer}, {\em A course on rough paths}, Springer, 2020.

\bibitem{friz2018differential}
{\sc P.~K. Friz and H.~Zhang}, {\em Differential equations driven by rough
  paths with jumps}, Journal of Differential Equations, 264 (2018),
  pp.~6226--6301.

\bibitem{gershgorin1931uber}
{\sc S.~A. Gershgorin}, {\em \"{U}ber die {A}bgrenzung der {E}igenwerte einer
  matrix}, Proceedings of the USSR Academy of Sciences. VII series. Division of
  Mathematical and Natural Sciences,  (1931), pp.~749--754.

\bibitem{gladyshev1961new}
{\sc E.~Gladyshev}, {\em A new limit theorem for stochastic processes with
  {G}aussian increments}, Theory of Probability \& Its Applications, 6 (1961),
  pp.~52--61.

\bibitem{golub2013matrix}
{\sc G.~H. Golub and C.~F. Van~Loan}, {\em Matrix computations}, JHU press,
  2013.

\bibitem{gottlieb2001strong}
{\sc S.~Gottlieb, C.-W. Shu, and E.~Tadmor}, {\em Strong stability-preserving
  high-order time discretization methods}, SIAM review, 43 (2001), pp.~89--112.

\bibitem{Gub04}
{\sc M.~Gubinelli}, {\em Controlling rough paths}, Journal of Functional
  Analysis, 216 (2004), pp.~86--140.

\bibitem{guyon1989convergence}
{\sc X.~Guyon and J.~Le{\'o}n}, {\em Convergence en loi des {H}-variations d'un
  processus {G}aussien stationnaire sur r}, Annales de l'IHP Probabilit{\'e}s
  et statistiques, 25 (1989), pp.~265--282.

\bibitem{han2021hurst}
{\sc X.~Han and A.~Schied}, {\em The roughness exponent and its model-free
  estimation}, arXiv preprint arXiv:2111.10301,  (2021).

\bibitem{hofmanova2021rough}
{\sc M.~Hofmanov{\'a}, J.-M. Leahy, and T.~Nilssen}, {\em {On a rough
  perturbation of the Navier–Stokes system and its vorticity formulation}},
  The Annals of Applied Probability, 31 (2021), pp.~736 -- 777.

\bibitem{hussaini1987spectral}
{\sc M.~Y. Hussaini and T.~A. Zang}, {\em Spectral methods in fluid dynamics},
  Annual review of fluid mechanics, 19 (1987), pp.~339--367.

\bibitem{istas1997quadratic}
{\sc J.~Istas and G.~Lang}, {\em Quadratic variations and estimation of the
  local {H}{\"o}lder index of a {G}aussian process}, Annales de l'Institut
  Henri Poincare (B) probability and statistics, 33 (1997), pp.~407--436.

\bibitem{kidger2021on}
{\sc P.~Kidger}, {\em {O}n {N}eural {D}ifferential {E}quations}, PhD thesis,
  University of Oxford, 2021.

\bibitem{kidger2021equinox}
{\sc P.~Kidger and C.~Garcia}, {\em {E}quinox: neural networks in {JAX} via
  callable {P}y{T}rees and filtered transformations}, Differentiable
  Programming workshop at Neural Information Processing Systems 2021,  (2021).

\bibitem{klein1975quadratic}
{\sc R.~Klein and E.~Gin{\'e}}, {\em On quadratic variation of processes with
  {G}aussian increments}, The Annals of Probability,  (1975), pp.~716--721.

\bibitem{kubilius2012rate}
{\sc K.~Kubilius and Y.~Mishura}, {\em The rate of convergence of hurst index
  estimate for the stochastic differential equation}, Stochastic processes and
  their applications, 122 (2012), pp.~3718--3739.

\bibitem{kubilius2017parameter}
{\sc K.~Kubilius, Y.~Mishura, and K.~Ralchenko}, {\em Parameter estimation in
  fractional diffusion models}, vol.~8, Springer, 2017.

\bibitem{lacasce1996baroclinic}
{\sc J.~H. LaCasce}, {\em Baroclinic vortices over a sloping bottom}, PhD
  thesis, Massachusetts Institute of Technology, 1996.

\bibitem{lang2021pathwise}
{\sc O.~Lang and W.~Pan}, {\em A pathwise parameterisation for stochastic
  transport}, in Stochastic Transport in Upper Ocean Dynamics Annual Workshop,
  Springer International Publishing Cham, 2021, pp.~159--178.

\bibitem{liu2020discrete}
{\sc Y.~Liu and S.~Tindel}, {\em {Discrete rough paths and limit theorems}},
  Annales de l'Institut Henri Poincaré, Probabilités et Statistiques, 56
  (2020), pp.~1730 -- 1774.

\bibitem{liu2023power}
{\sc Y.~Liu and X.~Wang}, {\em Power variations and limit theorems for
  stochastic processes controlled by fractional {B}rownian motions}, 2023.

\bibitem{orszag1971elimination}
{\sc S.~A. Orszag}, {\em On the elimination of aliasing in finite-difference
  schemes by filtering high-wavenumber components.}, Journal of Atmospheric
  Sciences, 28 (1971), pp.~1074--1074.

\bibitem{10.1214/20-EJS1685}
{\sc F.~Panloup, S.~Tindel, and M.~Varvenne}, {\em {A general drift estimation
  procedure for stochastic differential equations with additive fractional
  noise}}, Electronic Journal of Statistics, 14 (2020), pp.~1075 -- 1136.

\bibitem{papavasiliou2022inverse}
{\sc A.~Papavasiliou, T.~Papamarkou, and Y.~Zhao}, {\em The inverse problem for
  controlled differential equations}, arXiv preprint arXiv:2201.10300,  (2022).

\bibitem{papavasiliou2016approximate}
{\sc A.~Papavasiliou and K.~B. Taylor}, {\em Approximate likelihood
  construction for rough differential equations}, arXiv preprint
  arXiv:1612.02536,  (2016).

\bibitem{redmann2022runge}
{\sc M.~Redmann and S.~Riedel}, {\em Runge-{K}utta methods for rough
  differential equations}, Journal of Stochastic Analysis, 3 (2022), p.~6.

\bibitem{Rogers}
{\sc L.~C.~G. Rogers}, {\em Arbitrage with fractional {B}rownian motion},
  Mathematical finance, 7 (1997), pp.~95--105.

\bibitem{vershynin2013}
{\sc M.~Rudelson and R.~Vershynin}, {\em {Hanson-Wright inequality and
  sub-gaussian concentration}}, Electronic Communications in Probability, 18
  (2013), pp.~1 -- 9.

\bibitem{SCHIED2016974}
{\sc A.~Schied}, {\em On a class of generalized takagi functions with linear
  pathwise quadratic variation}, Journal of Mathematical Analysis and
  Applications, 433 (2016), pp.~974--990.

\bibitem{Young36}
{\sc L.~C. Young}, {\em {An inequality of the {H}ölder type, connected with
  Stieltjes integration}}, Acta Mathematica, 67 (1936), pp.~251 -- 282.

\bibitem{zhao2018p}
{\sc Y.~Zhao}, {\em P-variation calculation and statistical inference for the
  discretely observed differential equation driven by fractional Brownian
  motion}, PhD thesis, University of Warwick, 2018.

\bibitem{zhou2018parameter}
{\sc H.~Zhou}, {\em Parameter estimation for stochastic differential equations
  driven by fractional {B}rownian motion}, PhD thesis, University of Kansas,
  2018.

\end{thebibliography}
\end{document}